\documentclass[
  a4paper, 
  reqno, 
  oneside, 
  12pt
]{amsart}

\usepackage[utf8]{inputenc}

\usepackage[dvipsnames]{xcolor} 
\colorlet{cite}{red}
\usepackage{tikz}  

\usetikzlibrary{arrows,positioning}
\tikzset{ 
  baseline=-2.3pt,
  text height=1.5ex, text depth=0.25ex,
  >=stealth,
  node distance=2cm,
  mid/.style={fill=white,inner sep=2.5pt},
}

\usepackage{lmodern}
\usepackage{tikz-cd}
\usepackage{amsthm, amssymb, amsfonts}

\usepackage{amsthm, amssymb, amsfonts}	
\usepackage[all]{xy}
\usepackage{graphicx,caption,subcaption}
\usepackage{braket}  
\usepackage{microtype}
\usepackage{DotArrow}
\usepackage[makeroom]{cancel}
\usepackage[%
  bookmarks=true,			
  unicode=true,			
  pdftitle={Novikov-Morse theory},		%
  pdfauthor={Valencia},	%
  pdfkeywords={Novikov-Morse theory, groupoids},	
  colorlinks=true,		
  linkcolor=black,			
  citecolor=black,		
  filecolor=magenta,		
  urlcolor=RoyalBlue			
]{hyperref}				
\usepackage{cleveref}

\newtheoremstyle{mydef}
  {}		
  {}		
  {}		
  {}		
  {\scshape}	
  {. }		
  { }		
  {\thmname{#1}\thmnumber{ #2}\thmnote{ #3}}	

\newtheorem{theorem}{Theorem}[section]
\newtheorem*{theorem*}{Theorem}
\newtheorem{proposition}[theorem]{Proposition}
\newtheorem*{proposition*}{Proposition}
\newtheorem{lemma}[theorem]{Lemma}
\newtheorem*{lemma*}{Lemma}
\newtheorem{corollary}[theorem]{Corollary}
\newtheorem*{corollary*}{Corollary}
\theoremstyle{definition}
\newtheorem{definition}[theorem]{Definition}
\newtheorem{example}[theorem]{Example}

\theoremstyle{remark}
\newtheorem{remark}[theorem]{Remark}

\newtheorem*{conjecture*}{Conjecture}
\usepackage{tikz}

\newcommand{\rr}{\rightrightarrows}

\author{Fabricio Valencia}
\address{}
\date{\today}
\dedicatory{To the memory of Joaquim Fernando Lazzaro Coelho}

\address{F. Valencia - Instituto de Matem\'atica e Estat\'istica, Universidade de S\~ao Paulo, Rua do Mat\~ao 1010, Cidade Universit\'aria, 05508-090 S\~ao Paulo - Brasil. \newline  
      \phantom{xx}
 fabricio.valencia@ime.usp.br}

\title{Novikov type inequalities for orbifolds}

\begin{document}
\maketitle

\begin{abstract}
We propose a natural extension of the Novikov numbers for the basic cohomology class of a closed basic $1$-form on a proper Lie groupoid of finitely generated type. As an application, we prove corresponding Novikov inequalities for compact orbifolds.
\end{abstract}

\tableofcontents
\section{Introduction}


Morse inequalities relate the number of critical points of a Morse function on a compact manifold $M$ to its Betti numbers. More precisely, these inequalities determine a topological lower bound for the number of critical points with fixed index of a Morse function in terms of homology invariants and allow one to recover the Euler characteristic of $M$. In the seminal works \cite{No1,No2} Novikov started a generalization of Morse theory in which instead of critical points of smooth functions he dealt with closed 1-forms and 
their zeros, thus obtaining inequalities that generalize the Morse inequalities. To do so, Novikov defined numbers $b_j(\xi)$ and $q_j(\xi)$ which depend on a real cohomology class $\xi\in H^1(M,\mathbb{R})$, proving that any closed 1-form $\omega$ with Morse-type zeros has at least $b_j(\xi)+q_j(\xi)+q_{j-1}(\xi)$ zeros of index $j$ where $\xi=[\omega]$. Of course, the case $\xi=0$ gives rise to the usual Morse inequalities. Nowadays, Novikov theory is widely known, as it provides several applications in geometry, topology, analysis, and dynamics. Even though those inequalities were initially stated by Novikov, their rigorous proof was given by Farber in \cite{Far5}. It is worth mentioning that the \emph{Novikov inequalities} are closely related to the problem of finding topological lower bounds for the numbers of zeros which have Morse closed 1-forms lying in a given cohomology class. In  symplectic geometry such a problem turns out to be equivalent to finding topological lower bounds for the number of zeros of symplectic vector fields. We recommend the reader to consult \cite{F,Pa} for specific details about the classical Novikov inequalities and the study of closed 1-forms of Morse type on compact manifolds.

The main purpose of this paper is to extend some of the ingredients of the classical Novikov theory to the realm of Lie groupoids and their differentiable stacks, thus providing a way to generalize the Novikov inequalities for compact orbifolds. In this spirit, we plan to deal with cohomology classes of closed basic 1-forms on Lie groupoids whose orbits of zeros are nondegenerate in the sense of Morse--Bott theory. Before explaining our approach we may state the main result of this work in the following way:

\begin{theorem*}
Let $X$ be a compact orbifold and $\omega$ be a closed 1-form of Morse type on $X$. If $c_j(\omega)$ denotes the number of zeros of $\omega$ having Morse index $j$ then
$$c_j(\omega)\geq b_j(\xi)+q_j(\xi)+q_{j-1}(\xi), $$
where $\xi\in H^1(X,\mathbb{R})$ stands for the cohomology class of $\omega$.
\end{theorem*}

We will use proper étale groupoids
as orbifold atlases, so that we can think of an orbifold as being the orbit space of a Lie groupoid of this kind. Such a point of view turns out to be quite useful when dealing with several geometric and algebraic notions concerning the structure of an orbifold, or even a singular orbit space in general. In order to make sense of every term in the statement of our main result we need to introduce some terminology. Let $G\rr M$ be a Lie groupoid and let $\Pi_1(G)\rr M$ denote its fundamental groupoid of $G$-homotopy classes of $G$-paths, see \cite{BH,MoMr}. The fundamental group $\Pi_1(G, x_0)$ of $G$ with respect to a base-point $x_0\in M$ is defined to be the isotropy group $\Pi_1(G)_{x_0}$ which consists of $G$-homotopy classes of $G$-loops at $x_0$. Let us denote by $H_\bullet(G,\mathbb{Z})$ the total singular homology of $G$. There is a canonical way to define a Hurewicz $G$-homomorphism $h:\Pi_1(G, x_0)\to H_1(G,\mathbb{Z})$ which restricts as an isomorphism to the abelianization of $\Pi_1(G, x_0)$. Let $\omega$ be a closed basic $1$-form on $M$. That is, $\omega$ is a closed $1$-form on $M$ satisfying $(t^\ast - s^\ast)(\omega)=0$, so that this is the same as saying that $\omega$ is a simplicially closed 1-form on $G\rr M$ with respect to the Bott-Shulman-Stasheff double complex. For each smooth $G$-path $\sigma=\sigma_ng_n\sigma_{n-1}\cdots \sigma_1g_1\sigma_0$ in $M$ we can naturally define the $G$-path integral $\int_{\sigma}\omega=\sum_{k=0}^{n}\int_{\sigma_k}\omega$. It canonically determines a group homomorphism $l_{\omega}:\Pi_1(G,x_0)\to (\mathbb{R},+)$ which factors through the Hurewicz $G$-homomorphism $h$ by a uniquely determined group homomorphism $\textnormal{Per}_\xi: H_1(G,\mathbb{Z})\to (\mathbb{R}, +)$ which only depends on the basic cohomology class $\xi:=[\omega]$. This will be called the $G$-homomorphism of periods of $\xi$. Let \textbf{Nov} denote the Novikov ring of $\mathbb{R}$ and let us further assume that our Lie groupoid $G\rr M$ is proper with compact orbit space $M/G$. We can define a ring homomorphism $\phi_\xi: \mathbb{Z}(\Pi_1(G,x_0))\to \textbf{Nov}$ by setting $\phi_\xi([\sigma]):=\tau^{\textnormal{Per}_\xi(h([\sigma]))}$ for all $[\sigma]\in \Pi_1(G,x_0)$. It yields a local system $\mathcal{L}_\xi$ of left $\textbf{Nov}$-modules over $G\rr M$ whose total homology groups $H_j^{\textnormal{tot}}(G,\mathcal{L}_\xi)$ are called Novikov homology groups of $\xi$. When the homology $H_j^{\textnormal{tot}}(G,\mathcal{L}_\xi)$ is a finitely generated module over the ring {\bf Nov} we get that it is a direct sum of a free submodule with a torsion submodule since {\bf Nov} is a principal ideal domain. In this case, the Novikov Betti number $b_j(\xi)$ is defined to be the rank of the  free summand of $H_j^{\textnormal{tot}}(G,\mathcal{L}_\xi)$ and the Novikov torsion number $q_j(\xi)$ is defined to be the minimal number of generators of the torsion submodule of $H_j^{\textnormal{tot}}(G,\mathcal{L}_\xi)$. Lie groupoids for which the homology $H_j^{\textnormal{tot}}(G,\mathcal{L}_\xi)$ is a finitely generated module over {\bf Nov} are particular cases of what we shall call Lie groupoids of finitely generated type.

The set of zeros of a closed basic $1$-form $\omega$ on $M$ is saturated so that it is formed by a disjoint union of groupoid orbits. Therefore, we say that an orbit $\mathcal{O}_x$ of zeros of $\omega$ is nondegenerate if its normal Hessian is a nondegenerate fiberwise bilinear symmetric form on $\nu(O_x)$. Accordingly, $\omega$ is said to be of Morse type if all of its orbits of zeros are nondegenerate. Let us fix a Riemannian groupoid metric on $G\rr M$ in the sense of del Hoyo and Fernandes \cite{dHF}. The nondegeneracy requirement over $\mathcal{O}_x$ imposed above allows us to use the groupoid metric to split $\nu(\mathcal{O}_x)$ into the Whitney sum of two subbundles $\nu_-(\mathcal{O}_x)\oplus \nu_+(\mathcal{O}_x)$ such that the normal Hessian of $\omega$ around $\mathcal{O}_x$ is strictly negative on $\nu_-(\mathcal{O}_x)$ and strictly positive on $\nu_+(\mathcal{O}_x)$. Let $G_x$ be the isotropy group at $x$. From \cite{OV} we know that the normal Hessian is invariant with respect to the normal representation $G_x\curvearrowright \nu(\mathcal{O}_x)_x$ so that it preserves the splitting above since the normal representation is by isometries in this case. In consequence, we get a normal sub-representation $G_x\curvearrowright \nu_-(\mathcal{O}_x)_x$. The index of $\mathcal{O}_x$ in $M/G$ is defined to be $\dim \nu_-(\mathcal{O}_{x})_x/G_{x}=\dim \nu_-(\mathcal{O}_{x})_x-\dim G_{x}$. It follows that closed basic $1$-forms $\omega$ on $M$ allow us to speak about stacky zeros of closed stacky 1-forms $\overline{\omega}$ in the differentiable stack $[M/G]$ presented by $G\rr M$, so that the previous expression gives rise to a notion of stacky index for the nondegenerate stacky zeros of $\overline{\omega}$, provided $\omega$ is of Morse type.

After choosing a proper étale Lie groupoid $G\rr M$ with compact orbit space $X=M/G$ the statement of our main result becomes clear. Additionally, some observations regarding the generality of our approach come in order below.

\begin{remark}

First, even though our approach seems to be applicable to more general cases, we can guarantee that the total singular homology $H_\bullet(G,\mathbb{Z})$ of $G\rr M$ recovers the singular homology $H_\bullet(M/G,\mathbb{Z})$ of the orbit space $M/G$ only for some particular cases, including orbifolds. This is because we shall need the additional assumption asking for the basic cohomology of $G$ to be isomorphic to the total the de Rham cohomology of $G$. Second, the general results and constructions that we introduce for Lie groupoids can be adapted to study other interesting features of closed basic 1-forms of Morse type. For instance, Morse theory of basic Harmonic forms as well as the topology of singular foliations of closed basic 1-forms over Riemannian Lie groupoids can be studied by using our techniques. This will be the content of the forthcoming work \cite{LGV}. Third, there exists a one-to-one correspondence between zeros of closed basic 1-forms and zeros of symplectic vector fields on $0$-symplectic groupoids in the sense of Hoffman--Sjamaar \cite{hsz}. Thus, the natural generalization of the Novikov theory that we have described above provides a tool for using topological methods to study zeros of symplectic vector fields on symplectic orbifolds.
\end{remark}

The paper is structured as follows. In Section \ref{S:2}, we show a $G$-version of the Hurewicz homomorphism and exhibit how to construct a groupoid covering space over $G\rr M$ out of a fixed subgroup in $\Pi_1(G, x_0)$. This is the content of Propositions \ref{Hur} and \ref{Covering}, respectively. Motivated by the notion of local coefficients system for topological spaces we introduce a corresponding notion as well as its associated homology in the groupoid setting. Indeed, we define a double chain complex out of a local system of modules over a Lie groupoid, thus proving that its associated total homology is Morita invariant, compare Lemma \ref{LemDoubleComplex} and Proposition \ref{MoritaLocal}. We also show in Proposition \ref{EilenbergProp} a $G$-version of the Eilenberg homology isomorphism.  In Section \ref{S:3} we study closed basic $1$-forms. We define the $G$-homomorphism of periods associated to the basic cohomology class $\xi$ of a closed basic $1$-form $\omega$ on $M$ and explore some of its elementary properties. In particular, we characterize when $\xi$ is an integral class and describe the groupoid covering space associated to the kernel of the $G$-homomorphism of periods of $\xi$, see Propositions \ref{IntegralClass} and \ref{FormulaEqui}. We introduce closed basic $1$-forms of Morse type and prove that such a notion is Morita invariant, see Proposition \ref{MoritaMorse}. This gives rise to a notion of stacky closed $1$-form of Morse type over the differentiable stack $[M/G]$ presented by $G\rr M$. Finally, in Section \ref{S:4} we present three different manners to define the Novikov numbers $b_j(\xi)$ and $q_j(\xi)$ associated to the basic cohomology class $\xi$ over a proper Lie groupoid of finitely generated type. As a result we show the Novikov inequalities for compact orbifolds, compare Theorem \ref{ThmNokivovOrbifolds}, and quickly explain how to use this result in order to find a lower bound for the amount of zeros of certain symplectic vector fields on symplectic orbifolds.

\vspace{.2cm}
{\bf Acknowledgments:} Part of this work was carried out during a visit to the Dipartimento di Matematica, Universit\`a degli Studi di Salerno, Fisciano, Italy in 2023.  I am very thankful for the hospitality and support that the Geometry Group gave me while being there. I have benefited from several conversations with Juan Camilo Arias, Daniel L\'opez-Garcia, Antonio Maglio, and Cristi\'an Ortiz. I am grateful to Luca Vitagliano for reading earlier versions of this work and for his helpful feedback. I am indebted to the anonymous referee because all her/his comments and suggestions greatly improved this work. Valencia was supported by Grants 2020/07704-7, 2022/11994-6, and 2024/14883-6  Sao Paulo Research Foundation - FAPESP.

\section{Some algebraic topology ingredients}\label{S:2}

Before defining the $G$-homomorphism of periods of a closed basic $1$-form as well as its corresponding Novikov numbers we have to fill out some algebraic topology gaps concerning the fundamental Lie groupoid of $G$-homotopy classes of $G$-paths  which, to our knowledge, do not seem to have been described before in the literature. We will assume that the reader is familiar with the notion of Lie groupoid and the geometric/topological aspects underlying its structure \cite{dH,MoMr}. Throughout this paper the structural maps of a Lie groupoid $G\rr M$ are denoted by $(s,t,m,u,i)$ where $s,t:G\to M$ are the maps respectively indicating the source and target of the arrows, $m:G^{(2)}\to G$ stands for the partial composition of arrows, $u:M\to G$ is the unit map, and $i:G\to G$ is the map determined by the inversion of arrows. The fundamental groupoid of $G$-homotopy classes of $G$-paths is denoted by  $\Pi_1(G)\rr M$ and its isotropy group at $x_0\in M$ is denoted by $\Pi_1(G, x_0)$ and called the \emph{fundamental group} of $G$ with respect to a base-point $x_0$, see \cite[s. 3.3]{MoMr} and \cite[c. G; s. 3]{BH}. Out of the simplicial manifold structure associated to $G\rr M$ we can define a singular double chain complex denoted by $\lbrace C_\bullet(G^{(\bullet)}),\partial,d\rbrace$ whose total singular homology is denoted by $H_\bullet (G,\mathbb{Z})$. Additionally, the total cohomology of the standard Bott-Shulman-Stasheff double cochain complex $\lbrace \Omega_\bullet(G^{(\bullet)}),\overline{\partial},d_{dR}\rbrace$ of $G$ determines the de Rham cohomology $H^\bullet_{dR}(G)$ of $G$, compare \cite{Be,crainicMoer} and \cite[s. 2]{Dupont}.  It is important to mention that, on the one hand, all these algebraic notions are Morita invariant, so they give rise to corresponding notions for the differentiable stacks presented by the Lie groupoids with which we will work. On the other hand, the singular homology of associated to the simplicial structure of $G$ as well as the fundamental groups of $G$ coincide with those related notions for the underlying geometric realization of $G$.

\subsection{Hurewicz $G$-homomorphism and coverings}
Let $G\rr M$ be a $G$-connected Lie groupoid and fix $x_0\in M$. Consider $\Pi_1(G,x_0)$ and $H_1(G,\mathbb{Z})$ the fundamental group of $G$ with respect to the base-point $x_0$ and the first total singular homology group of $G$, respectively. It is simple to check that every $G$-loop at $x_0$ canonically defines a $1$-cycle in $\tilde{C}_1(G)=C_0(G)\oplus C_1(M)$. Hence, this suggests us to define the map $h:\Pi_1(G,x_0)\to H_1(G,\mathbb{Z})$ which sends the $G$-homotopy class $[\sigma]$ of a $G$-loop $\sigma$ at $x_0$ to its corresponding homology class $|\sigma|$.

\begin{lemma}\label{LemHur1}
Let $\sigma$ and $\sigma'$ be two $G$-paths. The following assertions hold true:
\begin{itemize}
\item $\sigma+\sigma^{-1}$ is equivalent to a boundary in $\tilde{C}_1(G)$, and
\item if $\sigma_0'(0)=\sigma_n(1)$ then $\sigma\ast \sigma'-\sigma-\sigma'$ is also equivalent to a boundary in $\tilde{C}_1(G)$.
\end{itemize}
\end{lemma}
\begin{proof}
Firstly, for every path $\sigma_j$ in $\sigma=\sigma_ng_n\sigma_{n-1}\cdots \sigma_1g_1\sigma_0$ we may argue as in Lemma 3.2 from \cite[p. 173]{Br} and for every arrow $g_j$ in $\sigma$ we consider the pair of composable arrows $(g_j,g_j^{-1})$. Secondly, observe that the expression $\sigma\ast \sigma'-\sigma-\sigma'$ equals
$$\sigma_n\ast\sigma_0'-\sigma_n-\sigma_0'+\textnormal{boundary}.$$ 

So, the last assertion follows directly as a consequence of Lemma 3.1 in \cite[p. 173]{Br}.
\end{proof}

As probably expected, the map $h$ can be thought of as a Hurewicz homomorphism in our context. Although such a homomorphism becomes the standard Hurewicz map at the level of the geometric realization of $G$ we decided to include a detailed proof of the result below at the level of groupoids. This is because some of the ideas and techniques used in such a proof will be instrumental for other results later on.

\begin{proposition}[Hurewicz $G$-homomorphism]\label{Hur}
	The map $h$ is a well defined group homomorphism which restricts to the abelianization $\Pi_1^{ab}(G,x_0)=\Pi_1(G,x_0)/[\Pi_1(G,x_0),\Pi_1(G,x_0)]
	$ as an isomorphism.
\end{proposition}

\begin{proof}
We shall follow the usual strategy used to prove the classical version of this result \cite[c.4; s.3]{Br}. Throughout the proof we shall fix $\sigma=\sigma_ng_n\sigma_{n-1}\cdots \sigma_1g_1\sigma_0$ a $G$-loop at $x_0$. In order to see that $h$ is well defined let us consider $\sigma'=\sigma_n'g_n'\sigma_{n-1}'\cdots \sigma_1'g_1'\sigma_0'$ another $G$-loop at $x_0$ which lies in the $G$-homotopy class of $\sigma$. If $\sigma$ and $\sigma'$ are equivalent then the assertion is trivial. Let us suppose then that there is a $G$-homotopy $D(\tau,\cdot)=D_n(\tau,\cdot)d_n(\tau)\cdots d_1(\tau)D_0(\tau,\cdot)$ from $\sigma$ to $\sigma'$. We split the square $[0,1]\times [0,1]$ into two 2-simplices with vertices $\lbrace (0,0), (0,1), (1,1)\rbrace$ and $\lbrace (0,0), (1,0), (1,1)\rbrace$, respectively, so that every homotopy $D_j$ can be
regarded as the sum of two singular 2-simplices over $M$. The diagonal from $(0,0)$ to $(1,1)$ will be denoted by $\Lambda$. Therefore, by using Lemma \ref{LemHur1} we get:
	\begin{eqnarray*}
		\delta D &=& \sum_{j=1}^n(t\circ d_j-s\circ d_j)+\sum_{j=0}^n(s\circ d_{j+1}-D_j|_{\Lambda}-\sigma_j)-\sum_{j=0}^n(\sigma_j'-D_j|_{\Lambda}-t\circ d_j)\\
		&-&\sum_{j=1}^n(g_j'- g_j)+\textnormal{constant}  =  \sigma - \sigma' + \textnormal{boundary},
	\end{eqnarray*}
where $s\circ d_{n+1}:=D_n|_{[0,1]\times 1}=x_0\ \textnormal{and}\ t\circ d_{0}:=D_0|_{[0,1]\times 0}=x_0$. That is, $\sigma$ and $\sigma'$ are homologous which means that $h$ is well defined. Again, as consequence of Lemma \ref{LemHur1} we have that if $[\sigma],[\sigma']\in \Pi_1(G,x_0)$ then 
	$$h([\sigma]\ast [\sigma'])=h([\sigma\ast \sigma'])=|\sigma\ast \sigma'|=|\sigma|+ |\sigma'|=h([\sigma])+h([\sigma]),$$
thus obtaining that $h$ is a group homomorphism. Note that if $\sigma^{-1}$ denotes the inverse $G$-loop of $\sigma$ then $h([\sigma^{-1}])=-|\sigma|$ since $h$ is a homomorphism. Hence, by Lemma \ref{LemHur1} it follows that the commutator subgroup of $\Pi_1(G,x_0)$ lies inside $\textnormal{ker}(h)$ so that we get another well defined group homomorphism $\tilde{h}:\Pi_1^{ab}(G,x_0)\to H_1(G,\mathbb{Z})$.
	
Since our groupoid is $G$-connected we have that for each $x\in M$ there exists some $G$-path $\lambda_x$ from our base-point $x_0$ to $x$. The constant path at $x$ is denoted by $c_x$. Let us construct the inverse group homomorphism of $\tilde{h}$. It is clear that we may think of every arrow $g\in G$ as a $G$-path connecting the constant paths at $s(g)$ and $t(g)$. Also, it is obvious that every path $\gamma:[0,1]\to M$ is equivalent to the $G$-path $c_{\gamma(0)}1_{\gamma(0)}\gamma 1_{\gamma(1)}c_{\gamma(1)}$. In consequence, after possibly reversing orientations and fixing an order, each singular groupoid $1$-chain $\alpha\in \tilde{C}_1(G)$ can be thought of as a formal sum of not necessarily different $G$-paths $\alpha=\sum \alpha_j$. We define the homomorphism $l:\tilde{C}_1(G)\to \Pi_1^{ab}(G,x_0)$ as $l(\alpha)=[\sum_{\ast} \lambda_{\alpha_j(1)}^{-1}*\alpha_j*\lambda_{\alpha_j(0)}]$, where $\sum_{\ast}$ denotes the concatenation of all the  $G$-loops at $x_0$ given by $\lambda_{\alpha_j(1)}^{-1}*\alpha_j*\lambda_{\alpha_j(0)}$. Note that the definition of $l$ does not depend on the choice of the $G$-paths $\lambda_{\alpha_j(0)}$ and $\lambda_{\alpha_j(1)}$ since we are going into $\Pi_1^{ab}(G,x_0)$ instead of $\Pi_1(G,x_0)$. We claim that $l$ takes the boundaries in $\tilde{C}_1(G)$ into $1\in \Pi_1^{ab}(G,x_0)$. Indeed, it is clear that it suffices to check this assertion when we apply the boundary operator $\delta$ over a pair of composable arrows $(g,g')\in G^{(2)}$, a path $\tilde{\sigma}:[0,1]\to G$ and a singular 2-simplex $\Sigma: \triangle_2\to M$. Firstly, $\delta (g,g')=g'-g'g+g$ so that
	\begin{eqnarray*}
		l(\delta (g,g')) & = & l(g')l(g)l(g'g)^{-1}=[\lambda_{s(g')}^{-1}*g'*\lambda_{t(g)}][\lambda_{t(g)}^{-1}*g*\lambda_{s(g)}][\lambda_{s(g')}^{-1}*g'g*\lambda_{s(g)}]^{-1}\\
		& = & [\lambda_{s(g')}^{-1}*g'g(g'g)^{-1}*\lambda_{s(g')}]]=[\textnormal{constant}]=1.
	\end{eqnarray*} 

	Secondly, $\delta \tilde{\sigma}= t\circ \tilde{\sigma} -s\circ \tilde{\sigma}-(\tilde{\sigma}(1)-\tilde{\sigma}(0))$. Note that 
	$$
	l(t\circ \tilde{\sigma}+\tilde{\sigma}(0)-s\circ \tilde{\sigma}-\tilde{\sigma}(1))=l(t\circ \tilde{\sigma})l(\tilde{\sigma}(0))l(s\circ \tilde{\sigma})^{-1}l(\tilde{\sigma}(1))^{-1}\\
	$$
	$$ =[\lambda_{t\circ \tilde{\sigma}(1)}^{-1}*t\circ \tilde{\sigma}*\lambda_{t\circ \tilde{\sigma}(0)}][\lambda_{t\circ \tilde{\sigma}(0)}^{-1}*\tilde{\sigma}(0)*\lambda_{s\circ \tilde{\sigma}(0)}] [\lambda_{s\circ \tilde{\sigma}(1)}^{-1}*s\circ \tilde{\sigma}*\lambda_{s\circ \tilde{\sigma}(0)}]^{-1}[\lambda_{t\circ \tilde{\sigma}(1)}^{-1}*\tilde{\sigma}(1)*\lambda_{s\circ \tilde{\sigma}(1)}]^{-1}$$
	$$ =[\lambda_{t\circ \tilde{\sigma}(1)}^{-1}*(t\circ \tilde{\sigma}) \tilde{\sigma}(0)(s\circ \tilde{\sigma})^{-1} \tilde{\sigma}(1)^{-1}*\lambda_{t\circ \tilde{\sigma}(1)}]=1,$$
	since $\tilde{\sigma}(1)$ is $G$-homotopic to the $G$-path $(t\circ \tilde{\sigma}) \tilde{\sigma}(0)(s\circ \tilde{\sigma})^{-1}$, see \cite[p. 191]{MoMr}. Thirdly, the case of the singular $2$-simplex $\Sigma$ over $M$ follows directly by Lemma 3.5 in \cite[p. 174]{Br}.
	
	The fact we just proved implies that $l$ descends to a well defined group homomorphism $\tilde{l}:H_1(G,\mathbb{Z})\to \Pi_1^{ab}(G,x_0)$. Furthermore, if $\sigma$ is a $G$-loop at $x_0$ then 
	$$(\tilde{l}\circ \tilde{h})([\sigma])=\tilde{l}(|\sigma|)=[\lambda_{x_0}^{-1}*\sigma*\lambda_{x_0} ]=[\sigma],$$
	since $\lambda_{x_0}$ may be chosen to be just the constant $G$-path at $x_0$. Let us now look at the opposite composition. Observe that the assignment $x\mapsto \lambda_{x}$ allows us to send singular groupoid $0$-simplices into singular groupoid $1$-simplices and this can be clearly extended to a homomorphism $\lambda: \tilde{C}_0(G)\to \tilde{C}_1(G)$. Suppose that $\alpha=\sum \alpha_j$ is a singular groupoid 1-chain. Then, by Lemma \ref{LemHur1} it holds that
	\begin{eqnarray*}
	(\tilde{h}\circ l)(\alpha) &=&\tilde{h}([\sum_\ast \lambda_{\alpha_j(1)}^{-1}*\alpha_j*\lambda_{\alpha_j(0)}])=\sum|\lambda_{\alpha_j(1)}^{-1}*\alpha_j*\lambda_{\alpha_j(0)}|\\
	&=& \sum|\lambda_{\alpha_j(1)}^{-1}+\alpha_j+\lambda_{\alpha_j(0)}|=\sum|\lambda_{\alpha_j(0)}+\alpha_j- \lambda_{\alpha_j(1)}|\\
	&=& |\sum \alpha_j+\lambda_{\sum (\alpha_j(1)-\alpha_j(0))}|=|\alpha+\lambda_{\delta \alpha}|.
	\end{eqnarray*}

Therefore, if $\alpha$ is a singular groupoid $1$-cycle it follows that $(\tilde{h}\circ l)(\alpha)=|\alpha|$. In particular, $(\tilde{h}\circ \tilde{l})(|\alpha|)=|\alpha|$, as desired.
\end{proof}

Let us now consider the notion of covering space over a groupoid as defined in \cite[s. 3.3]{MoMr} and \cite[c. G; s. 3]{BH}. A \emph{right action} of a Lie groupoid $G\rr M$ along a smooth map $p:E\to M$ is a smooth map $\cdot:E\times_MG\to E$ such that $p(e\cdot g)=s(g)$, $(e\cdot g)\cdot g'=e\cdot (gg')$ for all $(g,g')\in G^{(2)}$, and $e\cdot 1_x= e$ for all $x\in M$. In these terms, a \emph{covering space} $E$ over $G\rr M$ is a covering space $p:E\to M$ equipped with a right $G$-action $E\times_{M}G\to E$ along $p$. Morphisms between two covering spaces $E$ and $F$ over $G$ are equivariant maps $f:E\to F$. It is clear that any such morphism is necessarily a covering projection. We will denote by $\Gamma^G(E)$ the set of \emph{equivariant automorphisms} of $E$. The map $p$ extends to a Lie groupoid morphism $p:E\rtimes G\to G$ from the action groupoid $E\rtimes G$ into $G$ defined by $p(e,g)=g$, which covers the covering map $E\to M$. We say that $E$ is \emph{universal} if the action groupoid $E\rtimes G\rr E$ is $(E\rtimes G)$-connected and the fundamental groupoid $\Pi_1(E\rtimes G,e_0)$ is trivial for one (hence all) base-point $e_0\in E$. The latter kind of Lie groupoids shall be called \emph{simply $G$-connected}.

\begin{example}\label{ExaUnivesal}
	There is an explicit way to construct the universal covering space over $G\rr M$ when it is $G$-connected. Namely, the target fiber $\Pi_1(G)(-,x_0)$ at $x_0\in M$ of the fundamental groupoid $\Pi_1(G)$ of $G$ becomes a covering space over $M$ by restricting the source map $s:\Pi_1(G)(-,x_0)\to M$. In fact, this map is a left principal $\Pi_1(G, x_0)$-bundle. Also, there is a natural right $G$-action on $\Pi_1(G)(-,x_0)$ along $s$ given by $([\sigma],g)\to [\sigma g c_{s(g)}]$, where $c_{s(g)}$ is the constant map at $s(g)$,  which makes it into a covering space over $G$. This is actually a universal covering space over $G$, compare \cite[p. 612-613]{BH}.
\end{example}

It is well known that $G$-paths have a unique path lifting property, consult \cite[p. 191]{MoMr}. Let $e_0$ be a base-point in $E$, denote by $x_0 = p(e_0)$, and suppose
that $\sigma=\sigma_ng_n\sigma_{n-1}\cdots \sigma_1g_1\sigma_0$ is a $G$-path from $x_0$ to $x$. Then there are unique paths
$\tilde{\sigma_n}, \tilde{\sigma_{n-1}}, \cdots, \tilde{\sigma_0}$ in $E$ with $p_\ast(\tilde{\sigma_j}) = \sigma_j$, $\tilde{\sigma_0}(0)=e_0$ and $\tilde{\sigma_i}(0)g_j = \tilde{\sigma_{j-1}}(1)$. By setting $\tilde{\sigma_j}(0)=e_j$ and $(e_j,g_j):e_jg_j\to e_j$ the arrows in $E\rtimes G$ it follows that $\tilde{\sigma}=\tilde{\sigma_n}(e_n,g_n) \tilde{\sigma_{n-1}} \cdots (e_1,g_1)\tilde{\sigma_0}$ is the unique $(E\rtimes G)$-path starting at $e_0$ which projects onto $\sigma$. Since $G$-homotopic paths are in this way lifted to $(E\rtimes G)$-homotopic paths, it holds that any covering space over $G$ with a base-point $e_0$ as above has a natural fiber-wise action of $\Pi_1(G, x_0)$. Indeed, if $[\sigma]$ is the $G$-homotopy class of a $G$-loop at $x_0$ then we define $e_0\cdot [\sigma]:=	\tilde{\sigma_n}(1)$. This defines a right action of $\Pi_1(G, x_0)$ on $p^{-1}(x_0)$. As a consequence of the previous lifting property we get that the induced group homomorphism $\Pi_1(E\rtimes G, e_0)\to \Pi_1(G,x_0)$ is injective \cite[p. 611]{BH}. Let us assume that the action groupoid $E\rtimes G\rr E$ is $(E\rtimes G)$-connected. Thus, it is simple to check that the action of  $\Pi_1(G, x_0)$ on $p^{-1}(x_0)$ is transitive which implies that $p^{-1}(x_0)$ is homogeneous. That is, it is isomorphic to the quotient $\Pi_1(G, x_0)/p_\ast(\Pi_1(E\rtimes G, e_0))$ since the isotropy at $e_0$ agrees with $p_\ast(\Pi_1(E\rtimes G, e_0))$ by definition. Observe that every element $f\in \Gamma^G(E)$ induces an automorphism of $p^{-1}(x_0)$ when seeing it as a $\Pi_1(G, x_0)$-space. Indeed, if $F:E\rtimes G\to E\rtimes G$ is the groupoid automorphism $F(e,g)=(f(e),g)$ induced by $f$ then note that the  $(E\rtimes G)$-path $F_\ast(\tilde{\sigma})$ has initial point $f(e_0)$ and final point $f(e_0\cdot [\sigma])$ and satisfies $p_\ast(F_\ast(\tilde{\sigma}))=\sigma$, which means that it is also a lift of $\sigma$. Therefore, $f(e_0)\cdot [\sigma]=f(e_0\cdot [\sigma])$ by uniqueness. 

Just as in the classical case it follows that if our covering space $E$ over $G$ is universal then $\Gamma^G(E)\cong \Pi_1(G,x_0)$, see \cite[p. 612]{BH}. This important fact together with the previous comments allow us to prove the following result.

\begin{proposition}\label{Covering}
	If $H$ is a subgroup of $\Pi_1(G,x_0)$ then there exists a covering $r:F\to M$ over $G$ such that $r_\ast(\Pi_1(F\rtimes G, e_0))=H$.
\end{proposition}

\begin{proof}
	Let us fix a universal covering $p:E\to M$ over $G$, recall Example \ref{ExaUnivesal}. As $\Gamma^G(E)\cong \Pi_1(G,x_0)$ and $\Gamma^G(E)$ acts on the fibers $p^{-1}(x_0)$ without any fixed points then we may assume that $\Pi_1(G,x_0)$ acts on it without fixed points, compare \cite[p. 612]{BH}. Choose $e_0\in p^{-1}(x_0)$ and define the subgroup $H'$ of $\Gamma^G(E)$ as follows: $f\in H'$ if and only if there exists $[\sigma]\in H$ such that $f(e_0)=e_0\cdot [\sigma]$. It is simple to check that $H$ and $H'$ are isomorphic. As $H'$ is a subgroup of $\Gamma^G(E)$ it is actually a properly discontinuous subgroup of diffeomorphisms of $E$. Let us define $F$ as the quotient manifold $E/H'$ and denote by $q:E\to F$ the corresponding canonical projection. It is well known that this is also a covering space. Denote by $r:F\to M$ the induced smooth map defined as $r([e])=p(e)$. As $r\circ q=p$ it follows that $r$ is also a covering space. The right action of $G$ along $p:E\to M$ induces a well defined right action of $G$ along $q:E\to F$ as $[e]\cdot g:=[e\cdot g]$ which, in turn, induces another right action of $G$ along $r:F\to M$ by the same expression since the elements of $\Gamma^G(E)$ are $G$-equivariant. Therefore, we get that $\Pi_1(G,x_0)$ acts transitively on the right of the fibers $r^{-1}(x_0)$. Let $\tilde{x_0}=q(e_0)\in r^{-1}(x_0)$. Finally, by the construction of $F$ it holds that the isotropy group of the right $\Pi_1(G,x_0)$-action on $r^{-1}(x_0)$ corresponding to $\tilde{x_0}$ is precisely the subgroup $H$. That is, $r_\ast(\Pi_1(F\rtimes G, e_0))=H$.
\end{proof}

\subsection{Groupoid homology with local coefficients}

Motivated by the notion of local coefficients system for topological spaces (see for instance \cite[c. 6]{Wh}), we introduce a corresponding notion as well as its associated homology in the groupoid setting. Let $R$ be a ring. A \emph{local system of $R$-modules} over $G\rr M$ is defined as a function which assigns to any point $x_0\in M$ a left $R$-module $\mathcal{L}_x$ and to any continuous $G$-path $\sigma$ from $x$ to $y$ an $R$-homomorphism $\sigma_\#: \mathcal{L}_y\to \mathcal{L}_x$ such that the following conditions are satisfied:
\begin{itemize}
\item if $\sigma$ and $\sigma'$ are $G$-homotopic then $\sigma_\#=\sigma'_\#$,
\item if $\sigma$ is the constant $G$-path at $x_0$ then  $\sigma_\#: \mathcal{L}_{x_0}\to \mathcal{L}_{x_0}$ is the identity map, and
\item if $\sigma_0'(0)=\sigma_n(1)$ then $(\sigma\ast \sigma')_\#=\sigma_\#\circ\sigma'_\#$.
\end{itemize}

Compare with \cite{crainic,Tu}. Note that for any two point $x,y\in M$ it holds that $\mathcal{L}_x$ and $\mathcal{L}_y$ are isomorphic by $G$-connectedness. In particular, any $G$-loop $\sigma$ at $x_0$ determines an automorphism $\sigma_\#: \mathcal{L}_x\to \mathcal{L}_x$ which depends only on its $G$-homotopy class. This implies that the correspondence $[\sigma]\mapsto \sigma_\#$ induces an action of $\Pi_1(G,x_0)$ on $\mathcal{L}_{x_0}$, thus turning $\mathcal{L}_{x_0}$ into a left module over the group ring $R[\Pi_1(G,x_0)]$. A \emph{homomorphism} between two local systems of $R$-modules $\mathcal{L}$ and $\mathcal{L}'$ over $G$ is just natural transformation $\Phi:\mathcal{L}\to \mathcal{L}'$. Namely, for each $x_0\in M$ we have a homomorphism $\Phi_{x_0}:\mathcal{L}_{x_0}\to \mathcal{L}'_{x_0}$ such that for each $G$-path $\sigma$ from $x$ to $y$ it holds $(\sigma_\#)' \circ \Phi_y=\Phi_x\circ \sigma_\#$. If each $\Phi_x$ is an isomorphism then we say that $\mathcal{L}$ and $\mathcal{L}'$ are \emph{isomorphic}.

\begin{lemma}\label{LocalSystemCharacterized}
Suppose that our groupoid is $G$-connected.
\begin{itemize}
\item[i.] Let $\mathcal{L}$ and $\mathcal{L}'$ be two local systems of $R$-modules over $G\rr M$ and let $\phi: \mathcal{L}_{x_0}\to \mathcal{L}'_{x_0}$ be an isomorphism. Then there exists a unique isomorphism $\Phi:\mathcal{L}\to \mathcal{L}'$ such that $\Phi_{x_0}=\phi$.
\item[ii.] Let $\mathcal{L}_0$ be an $R$-module acted upon by $\Pi_1(G,x_0)$. Then there exists a local system of $R$-modules $\mathcal{L}$ such that $\mathcal{L}_{x_0}\cong \mathcal{L}_0$ (as $R[\Pi_1(G,x_0)]$-modules) and which induces the given action of $\Pi_1(G,x_0)$ on $\mathcal{L}_{x_0}$.
\end{itemize}
\end{lemma}
\begin{proof}
The proofs of these statements are straightforward adaptations of the proofs of Theorems 1.11 and 1.12 in \cite[p. 263]{Wh}, so that they are left as an exercise to the reader. 
\end{proof}
\begin{example}\label{ExampleNovikov}
Suppose that $\alpha:\mathbb{Z}[\Pi_1(G,x_0)]\to R$ is a ring homomorphism. We may view $R$ as a left $\mathbb{Z}[\Pi_1(G,x_0)]$-module with the action $[\sigma]\cdot r:= r\alpha([\sigma]^{-1})$. Such an action commutes with the standard left $R$-module structure on $R$. Therefore, by Lemma \ref{LocalSystemCharacterized} it follows that $\alpha$ determines a local system of $\mathcal{L}_\alpha$ of $R$-modules over $G\rr M$. Note that each module $(\mathcal{L}_\alpha)_x$ is isomorphic to $R$.
\end{example}

Let $\mathcal{L}$ be a local system of $R$-modules over $G$ and consider the nerve structure $\lbrace G^{(n)}, d^n_k\rbrace$ of $G$. Let $\lambda_n:G^{(n)}\to G^{(0)}$ denote the ``first vertex'' map $\lambda_n(g_1,\cdots,g_n)=s(g_n)$ with $\lambda_0=\textnormal{id}_{G^{(0)}}$. For each $n\geq 0$, we define $C^q(G^{(n)},\mathcal{L})$ as the set of all functions $c$ with the following properties:
\begin{enumerate}
\item[$a.$] for any $n$-singular $q$-simplex $\Sigma: \triangle_q\to G^{(n)}$  the value $c(\Sigma)$ is defined and belongs to the $R$-module $\mathcal{L}_{\lambda_n(\Sigma(a_0))}$. Here we assume that $a_0$ is the first vertex of the standard $q$-simplex $\triangle_q$ with vertices $a_0,a_1,\cdots, a_q$, and
\item[$b.$] the set of $n$-singular $q$-simpleces $\Sigma$ such that $c(\Sigma)\neq 0$ is finite.
\end{enumerate}

The elements of $C_q(G^{(n)},\mathcal{L})$ are called \emph{$n$-singular $q$-chains with coefficients in $\mathcal{L}$}. Note that any chain $c\in C_q(G^{(n)},\mathcal{L})$ can be formally written as a finite sum of the form $c=\sum_j v_j\cdot \Sigma_j$ where $\Sigma_j:\triangle_q\to G^{(n)}$ are $n$-singular $q$-simplices and $v_j\in \mathcal{L}_{\lambda_n(\Sigma_j(a_0))}$. That is, $c(\Sigma_j)=v_j$ and $c(\Sigma)=0$ for any $n$-singular $q$-simplex $\Sigma$ different from $\Sigma_j$. Let $c=v\cdot \Sigma$ be an elementary $n$-chain (i.e. the sum above contains only one term). In particular, $C_q(G^{(n)},\mathcal{L})$ carries a natural structure of left $R$-module.

On the one hand, we define the boundary homomorphism $\partial: C_q(G^{(n)},\mathcal{L})\to C_q(G^{(n-1)},\mathcal{L})$ as

$$\partial c=\partial(v\cdot \Sigma)=\sum_{k=0}^{n-1}(-1)^kv\cdot (d_k^n\circ \Sigma)+(-1)^{n} (g^\Sigma_n)^{-1}_\#(v)\cdot (d^n_n\circ \Sigma).$$

Here $g^\Sigma_n$ denotes the arrow determined by the $n$-projection of $\Sigma(a_0)\in G^{(n)}$ onto $G^{(1)}$ and $(g^\Sigma_n)^{-1}_\#$ is the inverse of the isomorphism $(g^\Sigma_n)_\#:\mathcal{L}_{t(g^\Sigma_n)}\to \mathcal{L}_{s(g^\Sigma_n)}$ which is induced by the arrow $g^\Sigma_n$, viewed as a $G$-path from $s(g^\Sigma_n)=\lambda_n(\Sigma(a_0))$ to $t(g^\Sigma_n)=s(g^\Sigma_{n-1})=\lambda_{n-1}(d^n_n\circ \Sigma(a_0))$. It is well known that $\partial^2=0$, compare \cite{crainic,crainicMoer, Tu}. On the other hand, we also define the boundary operator $\overline{d}: C_q(G^{(n)},\mathcal{L})\to C_{q-1}(G^{(n)},\mathcal{L})$ as
$$\overline{d}c=\overline{d}(v\cdot\Sigma)=\sigma_\#(v)\cdot d_0(\Sigma)+\sum_{j=1}^{q}(-1)^j v\cdot d_j(\Sigma).$$

In this case $d$ denotes the usual homology boundary operator and  $\sigma:[0,1]\to G^{(0)}$ is the path from $\lambda_n(\Sigma(a_1))$ to $\lambda_n(\Sigma(a_0))$ defined by $\sigma(\tau)=\lambda_n(\Sigma((1-\tau)a_1+\tau a_0))$, which has corresponding isomorphism $\sigma_\#:\mathcal{L}_{\lambda_n(\Sigma(a_0))}\to \mathcal{L}_{\lambda_n(\Sigma(a_1))}$. This operator also satisfies $\overline{d}^2=0$, see \cite[p. 266]{Wh}.

\begin{lemma}\label{LemDoubleComplex}
The boundary operators $\partial$ and $\overline{d}$ commute.
\end{lemma}
\begin{proof}
On the one side,
\begin{eqnarray*}
\overline{d}(\partial(c)) &= & \sum_{k=0}^{n-1}(-1)^k(\sigma_k)_\#(v)\cdot d_0(d_k^n\circ \Sigma))+ \sum_{k=0}^{n-1}\sum_{j=1}^{q}(-1)^{k+j}v\cdot d_j(d_k^n\circ \Sigma))\\
& + & (-1)^n((\sigma_n)_\#((g^\Sigma_n)^{-1}_\#(v)))(d_0(d^n_n\circ \Sigma))+\sum_{j=1}^{q}(-1)^{j+n}(g^\Sigma_n)^{-1}_\#(v)\cdot d_j(d_n^n\circ \Sigma)),
\end{eqnarray*}
where $\sigma_k(\tau)=(\lambda_{n-1}\circ d_k^n) (\Sigma((1-\tau)a_1+\tau a_0))$ for all $k=0,1,\cdots, n$. On the other side,
\begin{eqnarray*}
	\partial(\overline{d}(c)) &= & \sum_{k=0}^{n-1}(-1)^k\sigma_\#(v)\cdot d_k^n\circ d_0(\Sigma)+ \sum_{j=1}^{q}\sum_{k=0}^{n-1}(-1)^{j+k}v\cdot d_k^n\circ d_j(\Sigma)\\
	& + & (-1)^n(g^{d_0(\Sigma)}_n)^{-1}_\#(\sigma_\#(v))\cdot d^n_n\circ d_0(\Sigma)+\sum_{j=1}^{q}(-1)^{j+n}(g^{d_j(\Sigma)}_n)^{-1}_\#(v)\cdot d_n^n\circ d_j(\Sigma).
\end{eqnarray*}

Firstly, note that for all $k=0,1,\cdots, n-1$ it holds that $\lambda_{n-1}\circ d_k^n=\lambda_{n}$ so that the first two expressions in the right hand side of the above equalities agree. The second two expressions are exactly the same. Secondly, as $d_j(\Sigma)(a_0)=\Sigma(a_0)$ for all $j=1,\cdots, q$ then $g^{d_j(\Sigma)}_n=g^{\Sigma}_n$ which implies that the last expressions also agree. It remains to prove that $(\sigma_n)_\#\circ(g^\Sigma_n)^{-1}_\#=(g^{d_0(\Sigma)}_n)^{-1}_\#\circ(\sigma)_\#$. Observe that $\lambda_{n-1}(d^n_n\circ \Sigma(a_1))=\lambda_{n-1}(\Sigma(a_1))$ so that $\sigma_n\ast(g^\Sigma_n)^{-1}$ and $(g^{d_0(\Sigma)}_n)^{-1}\ast\sigma$ are two $G$-paths from $\lambda_{n-1}(\Sigma(a_1))$ to $\lambda_{n}(\Sigma(a_0))$. Define the path $\gamma:[0,1]\to G^{(1)}$ as $\gamma(\tau)=(\textnormal{pr}_n(\Sigma((1-\tau)a_1+\tau a_0)))^{-1}$. This is such that $\gamma(0)=(g^{d_0(\Sigma)}_n)^{-1}$ and $\gamma(1)=(g^\Sigma_n)^{-1}$. Also, $t\circ \gamma= \sigma$ and $s\circ \gamma= \sigma_n$. Therefore, since $\gamma(1)$ is $G$-homotopic to the $G$-path $(s\circ \gamma)^{-1} \gamma(0)(t\circ \gamma)$ (see \cite[p. 191]{MoMr}), it holds that the $G$-paths above are $G$-homotopic and the result follows as desired.
\end{proof}

The previous result motivates the following definition.

\begin{definition}
The total homology associated to the double chain complex $\lbrace C_\bullet(G^{(\bullet)},\mathcal{L}),\partial,\overline{d} \rbrace$ will be denoted by $H_\bullet^{\textnormal{tot}}(G,\mathcal{L})$ and called the \emph{groupoid homology of $G\rr M$ with local coefficients in $\mathcal{L}$}.
\end{definition}

It is clear that each $H_q^{\textnormal{tot}}(G,\mathcal{L})$ inherits a natural structure of left $R$-module. Let $\phi:(G\rr M)\to (G'\rr M')$ be a Lie groupoid morphism and suppose that $\mathcal{L}$ is a local system of $R$-modules over $G'$. By using the induced Lie groupoid morphism $\phi_\ast:(\Pi_1(G)\rr M)\to (\Pi_1(G')\rr M')$ it is possible to define another local system of $R$-modules $\phi^\ast \mathcal{L}$ over $G$ by setting $(\phi^\ast \mathcal{L})_x=\mathcal{L}_{\phi_0(x)}$. Let $c=\sum_j v_j\cdot \Sigma_j$ be an $n$-singular $q$-chain in $C_q(G^{(n)},\phi^\ast \mathcal{L})$. As $v_j\in (\phi^\ast \mathcal{L})_{\lambda_n(\Sigma_j)(a_0)}=\mathcal{L}_{\lambda'_n(\phi_n\circ \Sigma_j)(a_0)}$ then we have a well defined map $\phi_\ast:C_q(G^{(n)},\phi^\ast \mathcal{L})\to C_q(G'^{(n)},\mathcal{L})$ given as $\phi_\ast(\sum_j v_j\cdot \Sigma_j)=\sum_j v_j\cdot (\phi_n\circ\Sigma_j)$. It is simple to check that $\phi_\ast$ is a homomorphism of left $R$-modules which commutes with both pairs of boundary operators $\partial,\overline{d}$ and  $\partial',\overline{d}'$ so that we have a well defined homomorphism of left $R$-modules $\phi_\ast:H_\bullet^{\textnormal{tot}}(G,\phi^\ast \mathcal{L})\to H_\bullet^{\textnormal{tot}}(G',\mathcal{L})$. It is well known that if $\phi$ is a Morita map then horizontal homologies defined by $\partial$ and $\partial'$ are isomorphic, see \cite{crainic,crainicMoer} and \cite[p. 214]{MoMr}. Thus, by using a usual argument of spectral sequences we conclude that:
\begin{proposition}\label{MoritaLocal}
A Morita map $\phi:G\to G'$ induces an isomorphism $\phi_\ast:H_\bullet^{\textnormal{tot}}(G,\phi^\ast \mathcal{L})\to H_\bullet^{\textnormal{tot}}(G',\mathcal{L})$.
\end{proposition}

Another way to prove the previous result is commented in Remark \ref{MoritaInvariance&CohomologyVersion} below. Let us fix a universal covering $p:E\to M$ over $G$ and consider the corresponding action groupoid $E\rtimes G\rr E$, see Example \ref{ExaUnivesal}. Recall that  $p$ extends to a Lie groupoid morphism $p:E\rtimes G\to G$, defined by $p(e,g)=g$, which covers the covering map $E\to M$. Pick $x_0\in M$. We know that $\Gamma^G(E)\cong \Pi_1(G,x_0)$ where the isomorphism is as follows. The group $\Gamma^G(E)$ acts simply transitively on the fiber $p^{-1}(x_0)$. Take $e_0\in p^{-1}(x_0)$. Given $f\in \Gamma^G(E)$, let $\tilde{\sigma}$ be a $E\rtimes G$-path joining $f(e_0)$ with $e_0$. Define the map $\beta:\Gamma^G(E)\to \Pi_1(G,x_0)$ as $\beta(f)=[p_\ast(\tilde{\sigma})]$. Observe that $\beta$ does not depend on $\tilde{\sigma}$ since $\Pi_1(E\rtimes G,e_0)$ is trivial. This is the isomorphism we are interested in. For every $[\sigma]\in \Pi_1(G,x_0)$ we denote by $\beta_{[\sigma]}=\beta^{-1}([\sigma])$ the corresponding element in $\Gamma^G(E)$. That is, $[\sigma]$ can be presented by $p_\ast(\tilde{\sigma})$ where $\tilde{\sigma}$ is any $E\rtimes G$-path from $e_0$ to $\beta_{[\sigma]}(e_0)$.

We also consider the nerve structure $\lbrace E^{(n)}, (d^n_k)^E\rbrace$ of the action groupoid $E\rtimes G\rr E$
 as well as its total singular groupoid homology. The group $\Gamma^G(E)$ acts naturally on $C_q(E^{(n)},\mathbb{Z})$ as follows. Let $f_1:E\rtimes G\to  E\rtimes G$ be the groupoid automorphism $f_1(e,g)=(f_0(e),g)$ induced by $f_0=f\in \Gamma^G(E)$. So, $f\cdot \Sigma = f_n\circ \Sigma$ for each $\Sigma:\triangle_q\to E^{(n)}$ where $f_n:E^{(n)}\to E^{(n)}$ is the induced map along the nerve. Let $\mathcal{L}$ be a local system of $R$-modules over $G$. Denote by $\mathcal{L}_0:=\mathcal{L}_{x_0}$. It is clear that $\Pi_1(G,x_0)$ acts on $\mathcal{L}_0$ from the left. We transform this action into a right action of $\Gamma^G(E)$ on $\mathcal{L}_0$ by setting $v\cdot \beta_{[\sigma]}=(\sigma_\#)^{-1}(v)$. Let us denote by $\mathcal{L}_0\otimes_{\Gamma^G(E)} C_q(E^{(n)},\mathbb{Z})$  the quotient group of the tensor product $\mathcal{L}_0\otimes C_q(E^{(n)},\mathbb{Z})$ by the subgroup $Q_q(\mathcal{L}_0,E^{(n)})$ generated by all elements of the form $v\cdot f\otimes \Sigma - v\otimes f\cdot \Sigma$ with $v\in \mathcal{L}_0$, $f\in \Gamma^G(E)$, and $\Sigma\in C_q(E^{(n)},\mathbb{Z})$. Observe that, after naturally extending them to the tensor product, the simplicial boundary operator $\partial$ maps $Q_q(\mathcal{L}_0,E^{(n)})$ onto $Q_{q}(\mathcal{L}_0,E^{(n-1)})$ and the homology boundary operator $d$ maps $Q_q(\mathcal{L}_0,E^{(n)})$ onto $Q_{q-1}(\mathcal{L}_0,E^{(n)})$ so that they pass to the quotient, thus giving rise to well defined boundary operators $\partial:\mathcal{L}_0\otimes_{\Gamma^G(E)} C_q(E^{(n)},\mathbb{Z})\to \mathcal{L}_0\otimes_{\Gamma^G(E)} C_q(E^{(n-1)},\mathbb{Z})$ and $d:\mathcal{L}_0\otimes_{\Gamma^G(E)} C_q(E^{(n)},\mathbb{Z})\to \mathcal{L}_0\otimes_{\Gamma^G(E)} C_{q-1}(E^{(n)},\mathbb{Z})$. Hence, we have obtained a double chain complex $\lbrace\mathcal{L}_0\otimes_{\Gamma^G(E)} C_\bullet(E^{(\bullet)},\mathbb{Z}),\partial,d\rbrace$ whose associated total cohomology will be denoted by $H_\bullet(\mathcal{L}_0\otimes_{\Gamma^G(E)}E,\mathbb{Z})$. 
 
 We are now in conditions to prove our $G$-version of the well known Eilenberg homology isomorphism, namely:

\begin{proposition}[Eilenberg $G$-isomorphism]\label{EilenbergProp}
There exists a chain isomorphism between $\mathcal{L}_0\otimes_{\Gamma^G(E)} C_\bullet(E^{(\bullet)},\mathbb{Z})$ and $C_\bullet(G^{(\bullet)},\mathcal{L})$. In particular, the total homologies $H_\bullet(\mathcal{L}_0\otimes_{\Gamma^G(E)}E,\mathbb{Z})$ and $H_\bullet^{\textnormal{tot}}(G,\mathcal{L})$ are isomorphic. 
\end{proposition}

\begin{proof}
Let us first exhibit isomorphisms $\tilde{p}: \mathcal{L}_0\otimes_{\Gamma^G(E)} C_q(E^{(n)},\mathbb{Z})\to C_q(G^{(n)},\mathcal{L})$ for each $q,n\geq 0$. It is clear that the simply $E\rtimes G$-connectedness implies that for every $e\in E$ there exists a unique $E\rtimes G$-homotopy class $[\xi_{e}]$ presented by any $E\rtimes G$-path $\xi_{e}$ from $e$ to $e_0$. Denote by $p_n:E^{(n)}\to G^{(n)}$ the map induced by $p$ along the nerves. Let $v_0\in \mathcal{L}_0$, $\Sigma:\triangle_q\to E^{(n)}$, $\Lambda=p_n\circ \Sigma$, $e=\lambda^{E}_n(\Sigma(a_0))$, $x=\lambda^{G}_n(\Lambda(a_0))=p(e)$ and define $\tilde{p}_o: \mathcal{L}_0\otimes C_q(E^{(n)},\mathbb{Z})\to C_q(G^{(n)},\mathcal{L})$ by 
$$\tilde{p}_o(v_0\otimes \Sigma )= v\cdot \Lambda,$$ 
where $v=(p_\ast(\xi_{e}))_\#(v_0)\in \mathcal{L}_x$. Analogously to how it was proven in Theorem 3.4 from \cite[p. 279]{Wh} it is simple to check that $\tilde{p}_o Q_q(\mathcal{L}_0,E^{(n)})=0$, so that $\tilde{p}_o$ induces a homomorphism $\tilde{p}: \mathcal{L}_0\otimes_{\Gamma^G(E)} C_q(E^{(n)},\mathbb{Z})\to C_q(G^{(n)},\mathcal{L})$ which is actually an isomorphism. It remains to verify that these maps determine a chain map between the double chain complexes involved.

Firstly, recall that $d_j(\Sigma(a_0))=\Sigma(a_0)$ for all $1\leq j\leq q$ so that
$$\tilde{p}_o(d_j(v_0\otimes \Sigma))=\tilde{p}_o(v_0\otimes d_j(\Sigma))=v\cdot (p_n\circ (d_j(\Sigma)))=v\cdot d_j(\Lambda)=\overline{d}_j(\tilde{p}_o(v_0\otimes \Sigma)).$$

As usual, denote by $\sigma(\tau)=\lambda_n^E(\Sigma((1-\tau)a_1+\tau a_0))$. Note that $e'=\lambda_n^E(d_0(\Sigma)(a_0))=\lambda_n^E(\Sigma(a_1))$, thus obtaining $\xi_{e'}=\sigma\ast \xi_{e}$. Hence
$$\tilde{p}_o(d_0(v_0\otimes \Sigma))=\tilde{p}_o(v_0\otimes d_0(\Sigma))=v'\cdot (p_n\circ (d_0(\Sigma)))=v'\cdot d_0(\Lambda),$$
where
$$v'=(p_\ast(\xi_{e'}))_\#(v_0)=(p_\ast(\sigma)\ast p_\ast(\xi_{e}))_\#(v_0)=(p_\ast(\sigma))_\# (v).$$

But $p\circ \lambda_n^E=\lambda_n^G\circ p_n$ so that $p_\ast(\sigma)(\tau)=\lambda_n^G(\Lambda((1-\tau)a_1+\tau a_0))$. That is, $\tilde{p}_o(d_0(v_0\otimes \Sigma))=\overline{d}_0(\tilde{p}_o(v_0\otimes \Sigma))$. 

Secondly, observe that for $ 0\leq k\leq n-1$ we have
\begin{eqnarray*}
\tilde{p}_o(\partial_k^E(v_0\otimes \Sigma)) &=& \tilde{p}_o(v_0\otimes (d_k^n)^E\circ \Sigma)=v\cdot (p_{n-1}\circ (d_k^n)^E\circ \Sigma)\\
&=&v\cdot ((d_k^n)^G\circ p_{n}\circ \Sigma)=v\cdot (d_k^n)^G\circ\Lambda=\partial_k^G(\tilde{p}_o(v_0\otimes \Sigma)).
\end{eqnarray*}

We denote by $e^\Sigma_n$ the arrow determined by the $n$-projection of $\Sigma(a_0)\in E^{(n)}$ onto $E\rtimes G$. Observe that $e'=\lambda_{n-1}^E((d_n^n)^E\circ\Sigma)(a_0))=s(e^\Sigma_{n-1})=t(e^\Sigma_n)$ which implies that $\xi_{e'}=(e^\Sigma_n)^{-1}\ast \xi_{e}$. Thus,
$$\tilde{p}_o(\partial_n^E(v_0\otimes \Sigma))=\tilde{p}_o(v_0\otimes (d_n^n)^E\circ(\Sigma))=v'\cdot (p_{n-1}\circ ((d_n^n)^E\circ(\Sigma)))=v'\cdot (d_n^n)^G\circ\Lambda,$$
where
$$v'=(p_\ast(\xi_{e'}))_\#(v_0)=(p_\ast((e^\Sigma_n)^{-1})\ast p_\ast(\xi_{e}))_\#(v_0)=(p(e^\Sigma_n)^{-1})_\# (v)=((g^\Lambda_n)^{-1})_\# (v),$$
since $p\circ \textnormal{pr}_n^E=\textnormal{pr}_n^G\circ p_n$. In consequence, $\tilde{p}_o(\partial_n^E(v_0\otimes \Sigma))=\partial_n^G(\tilde{p}_o(v_0\otimes \Sigma))$. This completes the proof.
\end{proof}

We finish this section by commenting that:

\begin{remark}\label{MoritaInvariance&CohomologyVersion}
Firstly, this Eilenberg $G$-isomorphism also implies that the total groupoid homology with local coefficients in $\mathcal{L}$ is Morita invariant since the usual total singular groupoid homology is Morita invariant, consult \cite{Be}. Secondly, it is worth mentioning that similarly to how they were defined the total homology groups $H_\bullet^{\textnormal{tot}}(G,\mathcal{L})$ it is possible to define total cohomology groups $H^\bullet_{\textnormal{tot}}(G,\mathcal{L})$ as well as to prove a $G$-version of the Eilenberg cohomology isomorphism. This can be done by mimicking our approach together with the classical constructions, see \cite[p. 14]{F}. 
\end{remark}

\section{Closed basic 1-forms}\label{S:3}

The aim of this section is to introduce closed basic $1$-forms on Lie groupoids and study some of their properties. In particular, closed basic $1$-forms of Morse type are defined. We say that a Lie groupoid $G\rr M$ is \emph{proper} if the source/target map $(s,t):G\to M\times M$ is proper. In this case the groupoid orbits $\mathcal{O}_x$ are embedded in $M$, the isotropy groups $G_x$ are compact, and the orbit space $M/G$ is Hausdorff, second-countable, and paracompact \cite{dH}. Let $G\rr M$ be a proper groupoid and denote by $X=M/G$ its corresponding orbit space. A differential form $\omega$ on $M$ is said to be \emph{basic} if $s^\ast \omega=t^\ast \omega$, consult \cite{PPT,Wa}. The set of basic forms will be denoted by $\Omega_{\textnormal{bas}}^\bullet(G)$. It is clear that the de Rham exterior differential on $\Omega^\bullet(M)$ restricts to $\Omega_{\textnormal{bas}}^\bullet(G)$, thus yielding the so-called \emph{basic cohomology} $H_{\textnormal{bas}}^\bullet(G,\mathbb{R})$ of $G$. Such a cohomology is Morita invariant and also satisfies that $H^\bullet(X,\mathbb{R})\cong H_{\textnormal{bas}}^\bullet(G,\mathbb{R})$, where $H^\bullet(X,\mathbb{R})$ denotes the singular cohomology of $X$. Furthermore,  $X$ can be triangulated so that the basic cohomology $H_{\textnormal{bas}}^\bullet(G,\mathbb{R})$ becomes a finite dimensional vector space \cite{PPT}.

\subsection{G-homomorphisms of periods}

Let $\omega$ be a closed basic $1$-form on $G$ and let $\xi$ denote the basic cohomology class $[\omega]\in H_{\textnormal{bas}}^1(G,\mathbb{R})$. We are interested in studying some features of $\xi$ by defining its $G$-homomorphism of periods as well as its corresponding covering space. In order to do so we will make use of the topological ingredients described in the previous section. For each smooth $G$-path $\sigma=\sigma_ng_n\sigma_{n-1}\cdots \sigma_1g_1\sigma_0$ in $M$ we define the $G$-path integral
$$\int_{\sigma}\omega=\sum_{k=0}^{n}\int_{\sigma_k}\omega.$$
\begin{lemma}\label{Lem1}
Let $[\sigma]$ denote the $G$-homotopy class of $\sigma$. Then the expression 
$$\int_{[\sigma]}\omega=\int_{\sigma}\omega,$$
is well defined. 
\end{lemma}
\begin{proof}
Let us pick another $G$-path $\sigma'$ being $G$-homotopic to $\sigma$. If $\sigma$ and $\sigma'$ are equivalent then the assertion is trivial. Suppose then that there is a smooth $G$-homotopy $D(\tau,\cdot)=D_n(\tau,\cdot)d_n(\tau)\cdots d_1(\tau)D_0(\tau,\cdot)$ from $\sigma$ to $\sigma'$ with fixed endpoints $x$ and $y$. It suffices to check that the expression 
$$I_\tau=\int_{D(\tau,\cdot)}\omega= \sum_{k=0}^{n}\int_{0}^{1}\omega_{D_k(\tau,\nu)}(D_k'(\tau,\nu))d\nu,$$
does not depend on $\tau$ by differentiating it with respect to $\tau$. However, this computation can be verified in local coordinates as in the classical case by using the fact that $\omega$ is a closed basic $1$-form and the following identities are satisfied: $D_k(\tau,0)=t(d_k(\tau))$, $D_{k-1}(\tau,1)=s(d_k(\tau))$ for all $k=1,\cdots,n$ and $D_0(\tau,0)=x$, $D_n(\tau,1)=y$ for all $\tau\in [0,1]$.

\end{proof}

More importantly:

\begin{remark}\label{RemarkBasicCohomologous}
Suppose that $\omega$ and $\omega'$ are basic-cohomologous closed basic 1-forms. That is, there is a basic smooth function $f:M\to \mathbb{R}$ such that $\omega-\omega'=df$. By arguing as in Lemma \ref{FormulaEqui} below it is simple to check that $\int_\sigma (\omega-\omega')=\int_\sigma df=f(\sigma_n(1))-f(\sigma_0(0))$ since $f$ is basic. In consequence, the expression $\int_{[\sigma]}\xi=\int_{\sigma}\omega$ is also well defined when $\sigma$ is a $G$-loop.
\end{remark}
The first interesting consequence of the previous result is the following.
\begin{proposition}\label{BasicTrivial}
If $G\rr M$ is simply $G$-connected then $H_{\textnormal{bas}}^1(G,\mathbb{R})=0$.
\end{proposition}
\begin{proof}
	Let $\omega$ be a closed basic $1$-form and define $f(x)=\int_{\lambda_x}\omega$ where $\lambda_x$ is any $G$-path joining $x_0$ with $x$. This function is smooth and well defined since $\Pi_1(G,x_0)$ is trivial. Let $\gamma:[1,3]\to M$ be a smooth path such that $\gamma(2)=x$ and $\gamma'(2)=X_x\in T_xM$. Let $\lambda_{\gamma(1)}$ be a fixed $G$-path from $x_0$ to $\gamma(1)$ and consider the $G$-path $\lambda_{\gamma(\tau)}=\gamma|_{[1,\tau]}1_{\gamma(1)}\lambda_{\gamma(1)}$ for $1\leq\tau\leq 3$. Observe that $f(\gamma(\tau))=f(\gamma(1))+\int_1^{\tau}\omega_{\gamma(\nu)}(\gamma'(\nu)) d\nu$. Thus
$$
		df_x(X_x) = \frac{d}{d\tau}\left( f(\gamma(1))+\int_1^{\tau}\omega_{\gamma(\nu)}(\gamma'(\nu)) d\nu\right)|_{\tau=2}= \frac{d}{d\tau}\left(\int_1^{\tau}\omega_{\gamma(\nu)}(\gamma'(\nu)) d\nu\right)|_{\tau=2}=\omega_x(X_x).$$

Let us now check that $f$ is basic. If $g\in G$ then it follows that $\lambda_{t(g)}=  c_{t(g)}g\lambda_{s(g)}$, where $c_{t(g)}$ denotes the constant path at $t(g)$ and $\lambda_{s(g)}$ is any $G$-path from $x_0$ to $s(g)$, is a $G$-path joining $x_0$ with $t(g)$. Note that by definition $\int_{\lambda_{t(g)}}\omega=\int_{\lambda_{s(g)}}\omega$ since the $G$-path $c_{t(g)}g$ does not contribute to the $G$-path integral of the left hand side. Hence, $f(s(g))=f(t(g))$ as desired.
	
\end{proof}

If $\Pi_1(G,x_0)$ is the fundamental group of $G$ at the base-point $x_0\in M$ then from Lemma \ref{Lem1} we get a well defined group homomorphism $l_{\omega}:\Pi_1(G,x_0)\to (\mathbb{R},+)$ by sending $[\sigma]\mapsto \int_{\sigma}\omega$. Since $\mathbb{R}$ is abelian it follows that $l_\omega$ factors through the Hurewicz $G$-homomorphism $h$ from Proposition \ref{Hur} by a uniquely determined group homomorphism $\textnormal{Per}_\xi: H_1(G,\mathbb{Z})\to \mathbb{R}$ which only depends on the cohomology class $[\omega]\in H_{\textnormal{bas}}^1(G,\mathbb{R})$, see Remark \ref{RemarkBasicCohomologous}. This will be called the \emph{$G$-homomorphism of periods} of $\omega$. Because of the isomorphism $H^1_{dR}(G)\cong \textnormal{Hom}(H_1(G,\mathbb{\mathbb{Z}}),\mathbb{R})$ and $H_{\textnormal{bas}}^1(G,\mathbb{R})\hookrightarrow H^1_{dR}(G)$ it follows that the homomorphism of periods $\textnormal{Per}_\xi$ determines entirely the basic cohomology class $\xi$ and, moreover, any group homomorphism $H_1(G,\mathbb{Z})\to \mathbb{R}$ can be realized as the homomorphism of periods of a closed basic 1-form on $M$. The last assertion can be proven by arguing as in \cite[p. 164-165]{F} but using instead basic forms as in \cite[s. 8]{PPT}.

Let $d\theta$ denote the angle form on $S^1$. Here $\theta=\frac{1}{2\pi}\phi$ where $\phi$ is the angle multi-valued function on $S^1$. This is a closed 1-form without zeroes which can not be presented as the differential of a smooth function. The latter fact is consequence of the Stokes Theorem and the fact that $\int_{S^1}d\theta=1$. It is clear that if $f:M\to S^1$ is basic then $f^\ast(d\theta)$ becomes a closed basic $1$-form on $M$. Let us characterize the closed basic 1-forms that can be obtained in this way. 

\begin{proposition}\label{IntegralClass}
Let $\omega$ be a closed basic 1-form on $M$. Then $\omega=f^\ast(d\theta)$ where $f:M\to S^1$ is a smooth basic function if and only if the cohomology class $\xi=[\omega]\in H_{\textnormal{bas}}^1(G,\mathbb{R})$ is integral, that is, $\xi\in H_{\textnormal{bas}}^1(G,\mathbb{Z})=H_{\textnormal{bas}}^1(G,\mathbb{R})\cap H^1(M,\mathbb{Z})$.
\end{proposition}

\begin{proof}
We will mainly follow the classical proof of this result as in \cite[p. 37]{F}. Suppose that $\omega=f^\ast(d\theta)$ with $f:M\to S^1$ basic. Note that $f$ induces a Lie groupoid morphism $F:(G\rr M)\to (S^1\rr S^1)$ where $F$ is either given by $s^\ast f$ or $t^\ast f$. Therefore, if $\sigma=\sigma_ng_n\sigma_{n-1}\cdots \sigma_1g_1\sigma_0$ is a $G$-loop then $$f_\ast(\sigma)=f_\ast(\sigma_n)1_{f_\ast(\sigma_n)(0)}f_\ast(\sigma_{n-1})\cdots f_\ast(\sigma_1)1_{f_\ast(\sigma_1)(0)}f_\ast(\sigma_0),$$
turns out to be equivalent to a usual loop on $S^1$ (actually, we obtain a branch of loops formed by $f_\ast(\sigma_j)$ with $j=0,1,\cdots,n$). In consequence, the number
$$\int_{\sigma}\omega= \sum_{k=0}^{n}\int_{\sigma_k}f^\ast(d\theta)=\sum_{k=0}^{n}\int_{f_\ast(\sigma_k)}d\theta=\int_{f_\ast(\sigma)}d\theta\in \mathbb{Z},$$
is an integer since it agrees with the sum of the degrees of the loops $f_\ast(\sigma_j)$ on $S^1$, which clearly computes the degree of the whole branch loop $f_\ast(\sigma)$. Thus, we get that the $G$-homomorphism of periods of any closed basic $1$-form $\omega=f^\ast(d\theta)$ with $f:M\to S^1$ basic takes integral values so that its associated cohomology class lies inside $H_{\textnormal{bas}}^1(G,\mathbb{Z})$. Conversely, let us now suppose that all the $G$-periods associated to $\xi$ are integral. Fix a base point $x_0\in M$ and define $f(x)=\exp\left( 2\pi \sqrt{-1}\int_{\lambda_x}\omega\right)$, where $\lambda_x$ is any $G$-path joining $x_0$ with $x$. Note that the definition of $f$ does not depend on $\lambda_x$ since if $\lambda'_x$ is another $G$-path from $x_0$ to $x$ then for the $G$-loop at $x_0$ we get $\sigma= (\lambda'_x)^{-1} \ast \lambda_x$ and $\int_{\lambda_x}\omega-\int_{\lambda'_x}\omega=\int_{\sigma}\omega \in \mathbb{Z}$. By performing similar computations as those in the proof of Proposition \ref{BasicTrivial} it is simple to check that $f$ is a smooth basic function satisfying $\omega=f^\ast(d\theta)$.
\end{proof}

We need to obtain an additional property regarding the $G$-homomorphism of periods associated to the basic cohomology class $\xi$ of a closed basic $1$-form $\omega$ on $M$. By Proposition \ref{Covering} we can consider a covering space $p:M_{\xi}\to M$ over $G$ which corresponds to the kernel of the $G$-homomorphism of periods $\textnormal{Per}_\xi\circ h^{-1}:\Pi_1(G,x_0)\to \mathbb{R}$. Every $(M_{\xi}\rtimes G)$-loop in $M_\xi$ project by $p$ to a $G$-loop in $M$ with trivial periods with respect to $\omega$. Therefore, it follows that the pullback basic 1-form $p^\ast\omega$ is basic exact. That is, there is a basic function $f:M_\xi \to \mathbb{R}$ such that $p^\ast\omega=df$. Such $f$ can be defined as $f(e)=\int_{\lambda_e}p^\ast\omega$ where $\lambda_e$ is any $(M_{\xi}\rtimes G)$-path joining $e$ to a fixed base point in $M_{\xi}$. Recall that the free abelian group of equivariant covering transformations $\Gamma^G(E)$ acts by the left on $M_\xi$. Thus, the cohomology class $\xi$ determines an injective group homomorphism $\alpha_\xi:\Gamma^G(E)\to \mathbb{R}$ with image equal to the group of periods. Indeed, we define $\alpha_\xi$ through the composition $\textnormal{Per}_\xi\circ \beta^{-1}$ where $\beta:\Gamma^G(E)\to \Pi_1(G,x_0)$ as $\beta(f)=[p_\ast(\tilde{\sigma})]$ is the isomorphism previously described. Therefore, for $\varphi\in \Gamma^G(E)$ we choose $x\in M_\xi$ and any $(M_{\xi}\rtimes G)$-path $\tilde{\sigma}$ from $x$ to $\varphi(x)$. Note that $p_\ast(\tilde{\sigma})$ defines a $G$-loop at $p(x)$, so that it also defines a homology class in $|p_\ast(\tilde{\sigma})|\in H_1(G,\mathbb{Z})$. Hence, since $H^1_{dR}(G)\cong \textnormal{Hom}(H_1(G,\mathbb{\mathbb{Z}}),\mathbb{R})$ and $H_{\textnormal{bas}}^1(G,\mathbb{R})\hookrightarrow H^1_{dR}(G)$ it makes sense to set
\begin{equation}\label{RankDeck}
\alpha_\xi(\varphi)=\langle \xi,|p_\ast(\tilde{\sigma})|\rangle=\int_{p_\ast(\tilde{\sigma})}\omega\in \mathbb{R}.
\end{equation}

Such an expression does not depend on $x$ nor $\tilde{\sigma}$. It is clear that $df=p^\ast \omega$ is invariant by the action of $\Gamma^G(E)$ but $f$ is not, in fact:
\begin{lemma}\label{FormulaEqui}
The following formula holds true
$$f(\varphi(x))=f(x)+\alpha_\xi(\varphi),$$
for all $x\in M_\xi$ and $\varphi\in \Gamma^G(E)$.
\end{lemma}
\begin{proof}
Pick a $(M_{\xi}\rtimes G)$-path $\tilde{\sigma}=\tilde{\sigma}_n\tilde{g}_n\tilde{\sigma}_{n-1}\cdots \tilde{\sigma}_1\tilde{g}_1\tilde{\sigma}_0$ from $x$ to $\varphi(x)$. Then, since $f$ is basic we get

\begin{eqnarray*}
\int_{\tilde{\sigma}} p^\ast \omega & =& \int_{\tilde{\sigma}} df = \sum_{k=0}^{n}(f(\tilde{\sigma_k}(1))-f(\tilde{\sigma_k}(0)))\\
&= & f(\tilde{\sigma_n}(1))+\sum_{k=1}^n(f(t(\tilde{g}_k))-f(s(\tilde{g}_k)))-f(\tilde{\sigma_0}(0))= f(\varphi(x))-f(x).
\end{eqnarray*}

But $\int_{\tilde{\sigma}} p^\ast \omega=\int_{p_\ast \tilde{\sigma}} \omega =\alpha_\xi(\varphi)$, so that the formula follows.
\end{proof}

\subsection{Closed basic 1-forms of Morse type}
Let $\omega$ be a closed basic $1$-form on $G$ and let $\xi$ denote the basic cohomology class $[\omega]\in H_{\textnormal{bas}}^1(G,\mathbb{R})$. Note that $\xi=0$ if and only if there exists a basic smooth function $f:M\to \mathbb{R}$ such that $\omega=df$. Therefore, in this case the Morse theoretical features of $\omega$ on $X$ are the same as those described by means of $f$ in \cite{OV}. We will be mainly interested in studying the case $\xi\neq 0$.
\begin{lemma}
	The set of zeros of $\omega\in \Omega_{\textnormal{bas}}^1(G)$ is saturated in $M$. In particular, if $\omega_1=s^\ast \omega=t^\ast \omega$ then we have a topological subgroupoid $\textnormal{Zeros}(\omega_1)\rr \textnormal{Zeros}(\omega)$ of $G\rr M$.
\end{lemma}
\begin{proof}
	This easily follows from the fact that both $t$ and $s$ are surjective submersions and $\textnormal{Zeros}(\omega_1)=s^{-1}\textnormal{Zeros}(\omega)=t^{-1}\textnormal{Zeros}(\omega)$.
\end{proof}
If $\omega$ is a closed basic 1-form then by the Poincar\'e Lemma (see \cite[Lem. 8.5]{PPT}), it follows that for each groupoid orbit $\mathcal{O}$ there exists an open neighborhood $\mathcal{O}\subset U\subset M$ and a basic smooth function $f_U\in \Omega_{\textnormal{bas}}^0(G|_{U})$ such that $\omega|_U=df_U$. If $U$ is connected then the function $f_U$ is determined by $\omega|_U$ uniquely up to a constant. In particular, $\textnormal{Zeros}(\omega|_{U})=\textnormal{Crit}(f_U)$. The \emph{normal Hessian} of $\omega$ along an orbit of zeros $\mathcal{O}$ is defined to the the normal Hessian of $f_U$ along the critical orbit $\mathcal{O}$.

These elementary facts motivate the following definition.
\begin{definition}
	An orbit of zeros $\mathcal{O}$ of $\omega$ is said to be \emph{nondegenerate} if and only if $f_U$ is nondegenerate along $\mathcal{O}$ in the sense of Bott. Accordingly, we say that $\omega$ is \emph{Morse} if all of its orbits of zeros are nondegenerate.
\end{definition}

The notion of Morse--Bott function is classical and was initially introduced by Bott in \cite{Bo}. Our first key observation regarding this notion is that it is Morita invariant. The reader is recommended to consult \cite{dH,MoMr} for studying the basics on Morita equivalences in the realm of Lie groupoids.

\begin{proposition}\label{MoritaMorse}
	Suppose that $G$ an $G'$ are Morita equivalent Lie groupoids. If $G'$ admits a Morse closed basic $1$-form then so does $G$.
\end{proposition}
\begin{proof}
	Let $P$ be a principal bi-bundle between $G$ and $G'$ with anchor maps $a_l:P\to M$ and $a_r:P\to M'$. If $\omega'$ is a Morse closed basic $1$-form on $G'$ then there exists a unique closed basic $1$-form $\omega$ on $G$ such that $a_l^\ast(\omega)=a_r^\ast(\omega')$, see \cite{PPT,Wa}. This establishes a correspondence between orbits of zeros since both $a_l$ and $a_r$ are surjective submersions, compare Lemma 5.11 in \cite{PTW}. Furthermore, if $\mathcal{O'}$ and $\mathcal{O}$ are related orbits of zeros then there are connected neighborhoods $\mathcal{O'}\subseteq U'\subseteq M'$ and $\mathcal{O}\subseteq U\subseteq M$ together with basic functions $f'_{U'}$ and $f_{U}$ such that $a_l^\ast(f_{U})=a_r^\ast(f'_{U'})$. Hence, if $\mathcal{O'}$ is nondegenerate then so is $\mathcal{O}$ since both $a_l$ and $a_r$ are surjective submersions.
\end{proof}

Let $[M/G]$ denote the separated differentiable stack presented by $G\rr M$, compare \cite{BX,dH}. Observe that if $\omega$ is a basic $1$-form on $G$ then the expression $\overline{\omega}([x])=\omega(x)$ for $[x]\in X$ is well defined. In particular, this fact allows us to define a \emph{stacky closed $1$-form} on $[M/G]$ as an element $\overline{\omega}$ presented by a closed basic $1$-form $\omega$ on $G$. In consequence, we say that $[x]$ is a \emph{stacky zero} of $\overline{\omega}$ if and only if every point in the orbit $\mathcal{O}_x$ through $x$ is a zero of   $\omega$. Also, a stacky zero $[x]$ is nondegenerate if and only if $\mathcal{O}_x$ is nondegenerate for $\omega$.

\begin{definition}
	A stacky closed $1$-form $\overline{\omega}$ on $[M/G]$ is \emph{Morse} if all of its stacky zeros are nondegenerate. That is, it is presented by a closed basic $1$-form of Morse type $\omega$  on $G$.  
\end{definition}

It is well known that if $G\rr M$ is a Lie groupoid then its \emph{tangent groupoid} $TG\rightrightarrows TM$ is obtained by applying the tangent functor to each of its structural maps. If $\mathcal{O}_x \subset M$ is an orbit then we can restrict the groupoid structure to $G_{\mathcal{O}_x}=s^{-1}(\mathcal{O}_x)=t^{-1}(\mathcal{O}_x)$, thus obtaining a Lie subgroupoid $G_{\mathcal{O}_x}\rightrightarrows \mathcal{O}_x$ of $G\rightrightarrows M$. Furthermore, the Lie groupoid structure of $TG\rightrightarrows TM$ induces a Lie groupoid $\nu(G_{\mathcal{O}_x})\rightrightarrows \nu(\mathcal{O}_x)$ on the normal bundles, having the property that all of its structural maps are fiberwise isomorphisms. In particular, we have that $\overline{dt}\circ \overline{ds}^{-1}:s^*\nu(\mathcal{O}_x)\to t^*\nu(\mathcal{O}_x)$ defines a representation $(G_{\mathcal{O}_x}\rr \mathcal{O}_x)\curvearrowright(\nu(\mathcal{O}_x)\to \mathcal{O}_x)$. As a consequence, for every $x \in M$ the isotropy group $G_x$ has a canonical representation on the normal fiber $\nu(\mathcal{O}_x)_x$ called the \emph{normal representation} of $G_x$ on the normal direction.

Let $\mathcal{O}_x$ be a nondegenerate orbit of zeros of $\omega$ and let $\mathcal{O}_x\subset U\subset M$ and  $f_U:U\to \mathbb{R}$ respectively be an open neighborhood and basic smooth function such that $\omega|_U=df_U$. Let us also fix a groupoid metric on $G\rr M$ in the sense of del Hoyo and Fernandes \cite{dHF}. Since the normal Hessian $\textnormal{Hess}(f_U)$ is nondegenerate it follows that by using the groupoid metric the normal bundle $\nu(\mathcal{O}_x)$ splits into the Whitney sum of two subbundles $\nu_-(\mathcal{O}_x)\oplus \nu_+(\mathcal{O}_x)$ such that $\textnormal{Hess}(f_U)$ is strictly negative on $\nu_-(\mathcal{O}_x)$ and strictly positive on $\nu_+(\mathcal{O}_x)$. Let $G_x$ be the isotropy group at $x$. From Lemma 5.4 in \cite{OV} we know that $\textnormal{Hess}(f_U)$ is invariant with respect to the normal representation $G_x\curvearrowright \nu(\mathcal{O}_x)_x$ so that it preserves the splitting above since the normal representation is by isometries in this case. In consequence, we get a normal sub-representation $G_x\curvearrowright \nu_-(\mathcal{O}_x)_x$.

As consequence of Proposition 5.8 in \cite{OV} we can set up the following definition.

\begin{definition}
The \emph{stacky index} of $[x]$ is defined to be
$$\dim \nu_-(\mathcal{O}_{x})_x/G_{x}=\dim \nu_-(\mathcal{O}_{x})_x-\dim G_{x}.$$

Additionally, a zero $[x]$ of $\overline{\omega}$ is said to be \emph{orientable} if the action of $G_x$ on $\nu_-(\mathcal{O})_x$ is orientation-preserving.
\end{definition}

Stacky Morse functions $F:[M/G]\to \mathbb{R}$ were defined as well as studied in \cite{OV}. These are completely determined by basic functions $f:M\to \mathbb{R}$ whose critical orbits are nondegenerate in the sense of Bott. It is important to mention that the main results of classical Morse theory were extended to the context of Lie groupoids and their differentiable stacks by using such a notion of basic Morse--Bott function. For instance, Morse-like inequalities for the orbit space $X$ were obtained, provided it is compact. The case of orbifolds was first developed by Hepworth in \cite{H}. In both cases singular homology with real coefficients was considered. Nevertheless, such an approach turns out to be equivalent to considering integer coefficients after using the universal coefficient theorem for homology, compare \cite[p. 74]{BanHur}.

\begin{remark}
On the one hand, analogous results of local behavior, as the stacky Morse lemma and its consequences, can be analogously proven for stacky $1$-forms of Morse type on $[M/G]$. In particular, stacky zeros are isolated in $X$ and, if $X$ is compact, then they are a finite amount. On the other hand, we can guarantee the existence of stacky $1$-forms of Morse type in the following cases:
\begin{itemize}
\item the groupoid $G\rr M$ is proper with the canonical projection $\pi:M\to X$ being a proper map,
\item the groupoid $G\rr M$ is proper having either compact orbit space $X$ or else a finite number of Morita types, and
\item the groupoid $G\rr M$ is proper and \'etale, so that $[M/G]$ is an orbifold. By \emph{\'etale} we mean that either the source or the target map of $G$ is a local diffeomorphism.
\end{itemize}

All these assertions may be shown by using the stacky results in \cite{OV} and the results for orbifolds proven in \cite{H} together with the classical arguments developed in \cite[s. 1.4]{Pa}.
\end{remark}
\begin{remark}
As consequence of Proposition \ref{IntegralClass}, we have that the notions introduced above allow us to speak about circle-valued stacky Morse functions $[M/G]\to S^1$. Here $S^1$ stands for the differentiable stack presented by the unit Lie groupoid $S^1\rr S^1$.
\end{remark}

\section{The Novikov numbers}\label{S:4}

Our goal in this section is to exhibit a natural extension of the Novikov numbers associated to the basic cohomology class of a closed basic $1$-form on certain proper Lie groupoids. At the end we shall prove corresponding Novikov inequalities for compact orbifolds.

We start by introducing the so-called Novikov ring. Let $\Gamma$ be an additive subgroup of $(\mathbb{R},+)$. We denote by $\textbf{Nov}(\Gamma)$ the set of formal power series of the form $\sum_{\gamma\in \Gamma}n_{\gamma}\tau^{\gamma}$ where $\tau$ is a formal variable, the coefficients are integers $n_\gamma\in \mathbb{Z}$, and the exponents $\gamma$ belong to $\Gamma$ and satisfy the following condition:
$$\textnormal{for any}\ c\in\mathbb{R}\ \textnormal{the set}\ \lbrace \gamma\in \Gamma: n_{\gamma}\neq 0,\ \gamma>c\rbrace\ \textnormal{is finite}.$$

Equivalently, an element in $\textbf{Nov}(\Gamma)$ can be represented in the form $\sum_{i=0}^\infty n_{i}\tau^{\gamma_i}$ where $n_i\in \mathbb{Z}$, $\gamma_i\in \Gamma$ with $\gamma_1>\gamma_2>\gamma_3>\cdots$ and $\gamma_i$ tends to $-\infty$. It follows that $\textbf{Nov}(\Gamma)$ is a commutative ring which is called the \emph{Novikov ring} of $\Gamma$. For the particular case $\Gamma=\mathbb{R}$ we shall abbreviate the notation $\textbf{Nov}(\Gamma)$ to $\textbf{Nov}$. Some important properties are that $\textbf{Nov}(\Gamma)$ is a principal ideal domain, if $\mathbb{F}$ is a field then $\mathbb{F}\otimes \textbf{Nov}(\Gamma)$ is also a field, and if $\Gamma_1\subset \Gamma_2\subset \mathbb{R}$ are additive subgroups then $\textbf{Nov}(\Gamma_2)$, seen as an $\textbf{Nov}(\Gamma_1)$-module, has no torsion and is flat. In particular, $\textbf{Nov}$ is torsion free.

\begin{definition}
A Lie groupoid $G\rr M$ is said to be of \emph{finitely generated type} if the total homology $H_\bullet^{\textnormal{tot}}(G,\mathcal{L})$ is a finitely generated $R$-module for any local system of $R$-modules $\mathcal{L}$ over $G\rr M$.
\end{definition}

An important class of Lie groupoids of finitely generated type is given by Lie groupoids for which each manifold forming its nerve has finitely many cells in every dimension, in particular, compact Lie groupoids.

From now on we assume that $G\rr M$ is a proper Lie groupoid of finitely generated type with compact orbit space $X$. Let $\xi\in H_{\textnormal{bas}}^1(G,\mathbb{R})$ be any basic cohomology class and $\omega$ be a closed basic 1-form presenting it. We can define a ring homomorphism $\phi_\xi: \mathbb{Z}(\Pi_1(G,x_0))\to \textbf{Nov}$ by setting $\phi_\xi([\sigma]):=\tau^{\textnormal{Per}_\xi(h([\sigma]))}$ for all $[\sigma]\in \Pi_1(G,x_0)$ and then extending linearly. As consequence of Example \ref{ExampleNovikov} we get that $\phi_\xi$ determines a local system $\mathcal{L}_\xi$ of left $\textbf{Nov}$-modules over $G\rr M$. The groups $H_j^{\textnormal{tot}}(G,\mathcal{L}_\xi)$ are called \emph{Novikov homology groups} of $\xi$. It follows that the homology $H_j^{\textnormal{tot}}(G,\mathcal{L}_\xi)$ is a finitely generated module over the ring {\bf Nov}. Since {\bf Nov} is a principal ideal domain we have that the module $H_j^{\textnormal{tot}}(G,\mathcal{L}_\xi)$ is a direct sum of a free submodule with a torsion submodule. The \emph{Novikov Betti number} $b_j(\xi)$ is defined to be the rank of the  free summand of $H_j^{\textnormal{tot}}(G,\mathcal{L}_\xi)$ and the \emph{Novikov torsion number} $q_j(\xi)$ is defined to be the minimal number of generators of the torsion submodule of $H_j^{\textnormal{tot}}(G,\mathcal{L}_\xi)$.

\begin{remark}\label{xi0Remark}
If $\xi=0$ then $\mathcal{L}_0$ is constant since the induced isomorphism with respect to $\mathcal{L}_0$ of every $G$-path has to be the trivial one. In other words, the homomorphism of periods $\textnormal{Per}_0$ takes values in $\mathbb{Z}\subset \textbf{Nov}$ and the homology with local coefficients $\mathcal{L}_0$ becomes singular homology with coefficients in \textbf{Nov}. By Proposition \ref{EilenbergProp} we know that the total chain complex $\textbf{Nov}\otimes_{\mathbb{Z}(\Pi_1(G,x_0))} \tilde{C}_\bullet(E)$ agrees in this case with $\textbf{Nov}\otimes_{\mathbb{Z}} \tilde{C}_\bullet(G)$. In consequence, by using the universal coefficient theorem for homology as in Lemma 1.12 from \cite{F} we obtain that $b_j(0)$ coincides with the rank (``Betti numbers'') of $H_j(G,\mathbb{Z})$ and $q_j(0)$ equals the minimal number of generators of the torsion subgroup of $H_j(G,\mathbb{Z})$. This is because $\textbf{Nov}$ is torsion free. 
\end{remark}

We claim that there are other two possible ways to define the Novikov numbers associated to $\xi$. In order to do so, we define the \emph{rank} of the basic cohomology class $\xi$ as the rank of the image of $\textnormal{Per}_\xi$. This number will be denoted by $\textnormal{rank}(\xi)$. Note that by arguing as in the proof Proposition \ref{BasicTrivial} we may prove that $\textnormal{rank}(\xi)=0$ if and only if there is a basic function $f:M\to \mathbb{R}$ such that $\omega=df$. That is, $\textnormal{rank}(\xi)=0$ if and only if $\xi=0$. Furthermore:

\begin{remark}
If we consider a covering space $p:M_{\xi}\to M$ over $G$ which corresponds to the kernel of the $G$-homomorphism of periods $\textnormal{Per}_\xi$ then the expression \eqref{RankDeck} also shows that the rank of the group $\Gamma^G(M_\xi)$ equals the rank of the cohomology class $\xi$.
\end{remark}

Recall that the orbit space $X$ can be triangulated so that the basic cohomology $H_{\textnormal{bas}}^\bullet(G,\mathbb{R})$ becomes a finite dimensional vector space \cite{PPT}. Thus, as an important fact we have that:

\begin{proposition}\label{DenseRank1}
	The set of classes in $H_{\textnormal{bas}}^1(G,\mathbb{R})$ having rank $1$ is dense.
\end{proposition}
\begin{proof}
The proof of this result is similar to that of Corollary 2.2 in \cite[p. 38]{F} when considering instead $H_{\textnormal{bas}}^1(G,\mathbb{R})$ and $H_{\textnormal{bas}}^1(G,\mathbb{Z})$.
\end{proof}

Let $\Gamma$ be an additive subgroup of $(\mathbb{R},+)$. Recall that an element $x$ in $\textbf{Nov}(\Gamma)$ can be represented in the form $x=\sum_{i=0}^\infty n_{i}\tau^{\gamma_i}$ where $n_i\in \mathbb{Z}$, $\gamma_i\in \Gamma$ with $\gamma_1>\gamma_2>\gamma_3>\cdots$ and $\gamma_i$ tends to $-\infty$. We denote by $\mathbb{Z}[\textbf{Nov}(\Gamma)]$ the group ring consisting only of finite sums $x$ as above. This is a subring of $\textbf{Nov}(\Gamma)$. Let $S\subset \mathbb{Z}[\textbf{Nov}(\Gamma)]$ be the subset consisting of the elements in $\mathbb{Z}[\textbf{Nov}(\Gamma)]$ with leading term $1$, so that we have a canonical inclusion of the localized ring $\mathcal{R}(\Gamma):=S^{-1}\mathbb{Z}[\textbf{Nov}(\Gamma)]$ into the Novikov ring $\textbf{Nov}(\Gamma)$. This is well defined since the elements of $S$ are invertible in $\textbf{Nov}(\Gamma)$. The ring $\mathcal{R}(\Gamma)$ is also a principal ideal domain which will be called the \emph{rational part} of $\textbf{Nov}(\Gamma)$. Consult \cite[s. 1.3]{F} for specific details. 

Firstly, let us denote by $\Gamma_\xi$ the image inside $\mathbb{R}$ of the $G$-homomorphism of periods $\textnormal{Per}_{\xi}: H_1(G,\mathbb{Z})\to \mathbb{R}$. It follows that $\Gamma_\xi$ is a finitely generated free abelian group. The class $\xi$ determines a ring homomorphism $\psi_\xi: \mathbb{Z}(\Pi_1(G,x_0))\to \mathcal{R}(\Gamma_\xi)$ defined as $\psi_\xi([\sigma]):=\tau^{\langle \xi,h([\sigma])\rangle}$ for all $[\sigma]\in \Pi_1(G,x_0)$. As above, the 
homomorphism $\psi_\xi$ gives rise to a local system of left $\mathcal{R}(\Gamma_\xi)$-modules $\mathcal{M}_\xi$ over $G\rr M$ and the homology $H_j^{\textnormal{tot}}(G,\mathcal{M}_\xi)$ is a finitely generated module over the principal ideal domain $\mathcal{R}(\Gamma_\xi)$. From Corollary 1.12 in \cite{F} we obtain that $\textbf{Nov}(\Gamma_\xi)$ is flat over $\mathcal{R}(\Gamma_\xi)$ so that we may get an isomorphism $H_j^{\textnormal{tot}}(G,\mathcal{L}_\xi)\cong \textbf{Nov}(\Gamma_\xi)\otimes_{\mathcal{R}(\Gamma_\xi)} H_j^{\textnormal{tot}}(G,\mathcal{M}_\xi)$ after applying a usual argument of spectral sequences. This immediately implies that the rank of $H_j^{\textnormal{tot}}(G,\mathcal{M}_\xi)$ equals $b_j(\xi)$ and the minimal number of generators of the torsion submodule of $H_j^{\textnormal{tot}}(G,\mathcal{M}_\xi)$ agrees with $q_j(\xi)$.

Secondly, let us consider the covering space $p:M_{\xi}\to M$ over $G$ which corresponds to the kernel of the $G$-homomorphism of periods $\textnormal{Per}_{\xi}$, see Proposition \ref{Covering}. It is simple to check that a $G$-loop $\sigma$ in $M$ lifts to another $(M_\xi\rtimes G)$-loop in $M_\xi$ if and only if $\textnormal{Per}_{\xi}(|\sigma|)=\langle \xi, |\sigma|\rangle=0$, where $|\sigma|\in H_1(G,\mathbb{R})$ denotes the corresponding homology class of the $G$-loop $\sigma$. Thus, by the isomorphism theorem it follows that the group of covering 
transformations $\Gamma^G(M_\xi)$ can be naturally identified with $L_\xi = H_1(G,\mathbb{Z})/\textnormal{ker}({\xi})$ and $\textnormal{Per}_{\xi}$ yields an isomorphisms between the groups $L_\xi$ and $\Gamma_\xi$. Observe that after fixing a base for the free abelian group $L_\xi$ we may identify the group  ring $\Lambda_\xi=\mathbb{Z}[L_\xi]$ with the ring of  Laurent integral polynomials $\mathbb{Z}[T_1,\cdots,T_r,T_1^{-1},\cdots,T_r^{-1}]$. Let us denote by $w_1,\cdots,w_r$ the weights of the variables $T_1,\cdots,T_r$ which are determined by the $G$-homomorphism of periods $\textnormal{Per}_{\xi}:L_\xi\to \Gamma_\xi$. These weights are linearly independent over $\mathbb{Z}$ so that we may define the weight of a monomial 
$T_1^{n_1}\cdots T_r^{n_r}$ as $\sum n_jw_j$. Denote by $S_\xi\subset \Lambda_\xi$ the set consisting of the Laurent polynomials such that the monomial of maximal 
weight appearing in them has coefficient $1$. This is a multiplicative subset and the localized ring $S_\xi^{-1}\Lambda_\xi$ is a principal ideal domain since it is isomorphic to the rational subring $\mathcal{R}(\Gamma_\xi)$ of the Novikov ring, compare again \cite[s. 1.3]{F}.

Let $p_1:E\to M$ be the universal covering over $G$ (see Example \ref{ExaUnivesal}) and let $p_2:F\to M$ be the covering space over $G$ corresponding to the kernel $\textnormal{ker}(\xi)\subset H_1(G,\mathbb{Z})$, after mapping it to $\Pi_1(G,x_0)$ by using the Hurewicz $G$-homomorphism, see Proposition \ref{Covering}. These covering spaces give rise to the action groupoids $E\rtimes G\rr E$ and $F\rtimes G\rr F$. By viewing at the total groupoid homology with local coefficients in $\mathcal{M}_\xi$ as the total groupoid equivariant homology described in Proposition \ref{EilenbergProp} we get isomorphisms among the total chain complexes
$$\mathcal{R}(\Gamma_\xi)\otimes_{\mathbb{Z}(\Pi_1(G,x_0))} \tilde{C}_\bullet(E)\cong S_\xi^{-1}\Lambda_\xi\otimes_{\Lambda_\xi} \tilde{C}_\bullet(F)\cong S_\xi^{-1} \tilde{C}_\bullet(F).$$

It is important to notice that in the previous identifications we used the isomorphism $\textnormal{Per}_{\xi}:L_\xi\to \Gamma_\xi$. As localization is an exact functor we obtain that $H_\bullet^{\textnormal{tot}}(G,\mathcal{M}_\xi)\cong  S_\xi^{-1} H_\bullet(F,\mathbb{Z})$. Hence, the Novikov Betti number $b_j(\xi)$ coincides with the rank of $H_j(F,\mathbb{Z})$ and the Novikov torsion number $q_j(\xi)$ equals the minimal number of generators of the torsion submodule of the $S_\xi^{-1}\Lambda_\xi$-submodule $S_\xi^{-1} H_j(F,\mathbb{Z})$.

\begin{remark}\label{SameResults}
The reader probably already noticed that the definitions of the Novikov numbers provided above for this new setting became both natural and straightforward after having described the algebraic/differential topology notions from Sections \ref{S:2} and \ref{S:3}. It is left as an exercise to the reader to verify that similar results as those in Sections 1.5 and 1.6 from \cite{F} may be adapted in our context without so many changes along the proofs. In particular, we have that if $\xi_1,\xi_2\in H_{\textnormal{bas}}^1(G,\mathbb{R})$ are two basic cohomology classes such that $\textnormal{ker}(\xi_1)=\textnormal{ker}(\xi_2)$ then $b_j(\xi_1)=b_j(\xi_2)$ for all $j$. Also, $q_j(\xi_1)=q_j(\lambda\xi_2)$ for all $\lambda\in \mathbb{R}$ with $\lambda>0$.
\end{remark}

\subsection{Novikov inequalities}

A Lie groupoid is said to be \emph{étale} if its manifolds of arrows and objects have the same dimension. Let us further assume that $G\rr M$ is \'etale, so that $[M/G]$ is an orbifold \cite{Moerd}. In this specific case we get that $H_{dR}^\bullet(G)\cong H_{\textnormal{bas}}^\bullet(G,\mathbb{R})$ (see \cite{TuX}), and, in turn, $H_{dR}^\bullet(G)\cong H^\bullet(X,\mathbb{R})$. It follows that we may identify the total singular homology $H_1(G,\mathbb{Z})$ of $G$ with the singular homology $H_\bullet(X,\mathbb{Z})$ of $X$.

\begin{remark}
If $p:E\to M$ is any covering space over $G$ (e.g. the universal covering from Example \ref{ExaUnivesal}), then the action groupoid $E\rtimes G\rr E$ is a proper \'etale Lie groupoid as well, meaning that it also represents an orbifold \cite{Moerd}. On the one hand, by Remark \ref{xi0Remark}  we have that if $\xi=0$ then $b_j(0)$ and $q_j(0)$ respectively recover the corresponding Betti and torsion numbers of the orbit space $X$. On the other hand, the action groupoids $E\rtimes G\rr E$ and $F\rtimes G\rr F$ associated to the covering spaces $p_1$ and $p_2$ over $G$ mentioned above are also proper \'etale Lie groupoids. Therefore, it follows that after naturally adapting Corollaries 4.13 and 4.14 from \cite[p. 224]{MoMr} to the homology case we may think of the total homologies $H_\bullet^{\textnormal{tot}}(G,\mathcal{L}_\xi)$, $H_\bullet^{\textnormal{tot}}(G,\mathcal{M}_\xi)$, $H_\bullet(E,\mathbb{Z})$ and $H_\bullet(F,\mathbb{Z})$ as being respectively identified with the usual homologies of the orbit spaces $H_\bullet(X,\pi_\ast(\mathcal{L}_\xi))$, $H_\bullet(X,\pi_\ast(\mathcal{M}_\xi))$, $H_\bullet(E/E\rtimes G,\mathbb{Z})$ and $H_\bullet(F/F\rtimes G,\mathbb{Z})$ where $\pi:M\to X$ denotes the canonical orbit projection.
\end{remark}

We are now in conditions to prove the Novikov inequalities for compact orbifolds. This can be done by following the strategy described in \cite[s. 2.3]{F} step by step. It is worth mentioning that the inequalities below depend at some point on the usual Morse inequalities for orbifolds which were already proven in \cite{H} (see also \cite{OV}). Although the ideas of the proof are natural and straightforward adaptations of the classical ones, we will bring enough details in order to use most of the machinery introduced in the previous sections.

\begin{theorem}\label{ThmNokivovOrbifolds}
Let $G\rr M$ be a proper and \'etale Lie groupoid such that the orbit space $X$ is compact. Let $\omega$ be a Morse closed basic $1$-form on $M$. If $c_j(\omega)$ denotes the number of stacky zeros in $X$ having stacky Morse index $j$ then
\begin{equation}\label{NovikovInequalities}
c_j(\omega)\geq b_j(\xi)+q_j(\xi)+q_{j-1}(\xi),
\end{equation}
where $\xi=[\omega]\in H^1_{\textnormal{bas}}(G,\mathbb{R})$ is the basic cohomology class of $\omega$.
\end{theorem}
\begin{proof}
It suffices to prove the inequalities under the additional assumption that the basic cohomology class $\xi$ is integral. That is, $\xi\in H_{\textnormal{bas}}^1(G,\mathbb{Z})$. The latter requirement is equivalent to asking that $\xi$ has rank $1$. Indeed, basic cohomology classes $\xi$ of rank 1 are real multiples of integral basic cohomology classes, namely, $\xi = \lambda \xi_0$ where   $\xi\in H_{\textnormal{bas}}^1(G,\mathbb{Z})$ and $\lambda$ is a nonzero real number. It is because in this specific case the image of the $G$-homomorphism of periods $\textnormal{Per}_\xi$ is a cyclic subgroup in $\mathbb{R}$ so that all periods are integral multiples of a minimal period $\lambda\in \mathbb{R}_{>0}$. In other words, the $\lambda^{-1}\xi$ has all integral periods and it belongs to $H_{\textnormal{bas}}^1(G,\mathbb{Z})$. Therefore, if $\xi=\lambda \xi_0$ has rank $1$ and $\omega$ is a Morse closed basic 1-form in the class $\xi$ then $\omega_0=\lambda^{-1}\omega$ is another Morse closed basic 1-form in the class $\xi_0$ having the same zeros. In consequence, $b_j(\xi)=b_j(\xi_0)$ and $q_j(\xi)=q_j(\xi_0)$ by Remark \ref{SameResults}.

Assume for a moment that the Novikov inequalities \eqref{NovikovInequalities} hold true for basic cohomology classes of rank $1$. The argument to prove that the previous assumption is enough to ensure that the inequalities hold true for every basic cohomology class of rank $>1$ is similar to that in Lemma 2.5 from \cite{F} after considering instead basic cohomology. We sketch the argument here for the sake of completeness as well as the role it plays in our proof. Suppose that $\xi\in H_{\textnormal{bas}}^1(G,\mathbb{R})$ is a basic cohomology class of rank $>1$ and let $\omega$ be a Morse closed basic $1$-form in $\xi$. Let $S\subset X$ denote the set of zeroes of $\overline{\omega}$. This is a finite set since $X$ is compact. Consider the vector subspace $N_\xi=\lbrace \eta\in H_{\textnormal{bas}}^1(G,\mathbb{R}): \eta|_{\ker(\xi)}=0\rbrace$. By similar arguments as those used in the proof of Theorem 1.44 from \cite{F} and by Proposition \ref{DenseRank1} it follows that there exists a sequence of rank 1 basic cohomology classes $\xi_n\in N_\xi$ such that $\xi_n\to \xi$ as $n$ goes to $\infty$, $b_j(\xi_n)=b_j(\xi)$ and $q_j(\xi_n)=q_j(\xi)$ for all $j$ and $n$. Let us fix a basis $\eta_1,\cdots,\eta_r$ of $N_\xi$ whose elements are respectively represented by closed basic $1$-forms $\omega_1,\cdots,\omega_r$. Note that since $G$ is proper and $\eta_k|_{\ker(\xi)}=0$ we may ensure that there are open neighborhoods $U_k\subset M$ such that $U_k/G_{U_k}\subset X$ are open neighborhoods of $S$ and $\overline{\omega}_k$ vanishes identically on $U_k/G_{U_k}$ for all $k=1,\cdots,r$. It is clear that we can rewrite $\xi=\sum_{k=1}^r a_k \eta_k$ and $\xi_n=\sum_{k=1}^r a_{k,n} \eta_k$ with $a_k,a_{n,k}\in \mathbb{R}$ such that $a_{n,k}\to a_k$ as $n$ goes to $\infty$. Let us define $\omega_n=\omega-\sum_{k=1}^r(a_k-a_{n,k})\omega_k$. This is a basic closed $1$-form for which there is an open neighborhood $U\subset M$ made out of the $U_k's$ above such that $S\subset U/G_{U}$ and $\omega-\overline{\omega}_n$ vanishes identically on $U/G_{U}$. Furthermore, for $n$ large enough it follows that $\omega_n$ has no zeros outside $U$. That is, $c_j(\omega_n)=c_j(\omega)$ for any $j$ provided that $n$ is large enough. Note that from the defining formula above it follows that the basic cohomology class $[\omega_n]$ agrees with $\xi_n$ and they have rank $1$. In consequence, we get that if the Novikov inequalities hold true for $[\omega_n]$ then they must hold true also for $\xi$. 

Suppose that $\xi$ has rank $1$. Let us take a covering space $p:M_{\xi}\to M$ over $G$ which corresponds to the kernel of the $G$-homomorphism of periods $\textnormal{Per}_\xi\circ h^{-1}:H_1(G,\mathbb{Z})\to \mathbb{R}$ of the cohomology class $\xi$, see Proposition \ref{Covering}. In this case $M_\xi$ has an infinite cyclic group of equivariant covering transformations $\Gamma^G(M_\xi)$ whose generator is denoted by $T$. We already know that the pullback $p^\ast (\omega)$ is a closed basic-exact $1$-form so that there exists a basic function $f:M_\xi \to \mathbb{R}$, uniquely determined up to constant, such that $p^\ast (\omega)=df$. From Lemma \ref{FormulaEqui} it follows that $f(Tx)-f(x)=c$ is a constant for all $x\in M_\xi$. The number $c$ equals the minimal period of $\omega$ so that we may assume $c = \pm 1$ in our case since $\xi$ is integral. Assume that the generator $T$ is chosen so that $f(Tx)-f(x)=-1$ for all $x\in M_\xi$. Otherwise, we may take $T^{-1}$ instead of $T$.

Let us consider the action groupoid $M_\xi\rtimes G\rr M_\xi$, the stacky function $\overline{f}:[M_\xi/M_\xi\rtimes G]\to \mathbb{R}$ determined by $f$ and the Lie groupoid morphism $p:M_\xi\rtimes G\to G$ induced by the covering space $M_{\xi}\to M$ over $G$. We denote by $X_\xi$ the orbit space $M_\xi/M_\xi\rtimes G$ and by $\overline{T}:X_\xi\to X_\xi$ and $\overline{p}:X_\xi\to X$ the induced maps between the orbit spaces. On the one hand, observe that the formula $\overline{f}(\overline{T}[x])-\overline{f}([x])=-1$ holds true for all $[x]\in X_\xi$. On the other hand, the critical points of $f$ are precisely the elements in $p^{-1}(x)$ of the zeros $x\in M$ of $\omega$ so that the critical points of the stacky function $\overline{f}$ in $X_\xi$ are given by the elements in $\overline{p}^{-1}([x])$ of the zeroes $[x]$ of $\overline{\omega}$ in $X$. Pick a regular value $b\in \mathbb{R}$ of $\overline{f}$ and set $V=\overline{f}^{-1}(b)$, $N=\overline{f}^{-1}([b,b+1])$, and $Y=\overline{f}^{-1}((-\infty,b+1])$. Note that the projection $\overline{p}$ determines a one-to-one correspondence between the critical points of $\overline{f}|_{N}$ and the zeroes of $\overline{\omega}$. Furthermore, $\overline{f}|_{N}:N\to [b,b+1]$ is a stacky Morse function since $p_\xi$ is a local diffeomorphism and $\omega$ is of Morse type. Therefore, $c_j(\overline{f}|_{N})=c_j(\omega)$ for all $j=0,1,2,\cdots$. The homeomorphism $\overline{T}$ maps $Y$ into itself. As $X_\xi$ can be triangulated (see \cite{PPT}), it follows that we may fix a triangulation of $V$ which in turn induces another triangulation of $\overline{T}^{-1}V$, thus obtaining a simplicial isomorphism $\overline{T}:\overline{T}^{-1}V\to V$. Let us choose a triangulation of $N$ in such a way $V$ and $\overline{T}^{-1}V$ are sub-complexes. So, after applying the homeomorphism $\overline{T}$ we can get a triangulation of the whole $Y$ so that $\overline{T}:Y\to Y$ is represented by a simplicial map. In other words, we have obtained a chain complex $\overline{C}_\bullet(Y)$ of simplicial chains which actually is a complex of finitely generated $\mathbb{Z}[\overline{T}]$-modules.

The standard Morse inequalities for orbifolds were proved in Theorem 7.11 from \cite{H}. Hence, by mimicking the analysis of the Betti numbers $b_j(\overline{C},\mathfrak{p})$ associated to different prime ideals $\mathfrak{p}\subset\mathbb{Z}[\overline{T}]$, exactly to how it is done in the remaining part of the proof of Theorem 2.4 in \cite{F}, we can get the inequalities 
$$\sum_{k=0}^j(-1)^{k}c_{j-k}(\omega)\geq q_j(\xi)+\sum_{k=0}^j(-1)^kb_{j-k}(\xi),$$
which are slightly stronger than the Novikov inequalities we wanted to prove.
\end{proof}

In order to exhibit some examples where Theorem \ref{ThmNokivovOrbifolds} can be worked out we look at the special class of actions groupoids. Let us suppose that $K$ is a Lie group acting on a connected smooth manifold $M$, so that we can construct the action groupoid $K\ltimes M\rr M$. On the one hand, the set of $(K\ltimes M)$-loops based at $x\in M$ is in one-to-one correspondence with the set of pairs $(\sigma,k)$ where $\sigma:[0,1]\to M$ is a smooth path starting at $x$ and $k\in K$ is such that $k\sigma(0)=\sigma(1)$, compare \cite[p. 608]{BH} and \cite[p. 192]{MoMr}. Additionally, there exists a short exact sequence of groups:
\begin{equation}\label{actionGpaths}
 1\to \Pi_1(M,x)\to \Pi_1(K\ltimes M,x)\to \Pi_0(K)\to 1,
\end{equation}
where $\Pi_0(K)$ denotes the set of connected components of $K$. On the other hand, it is well known that a differential form on $M$ is \emph{basic} if it is $K$-invariant and horizontal, the latter term meaning that it vanishes on vectors tangent to the $K$-orbits. As shown in \cite{Wa}, a differential form is basic in the previous sense if and only if it is a basic differential form on the action groupoid $K\ltimes M\rr M$.

The equivariant version of the Novikov inequalities was obtained by Braverman--Farber in \cite{BravermanFarber}, provided that both $K$ and $M$ are compact. Therefore, as an immediate consequence we get:

\begin{corollary}
Let $K$ be a finite Lie group acting effectively on a compact manifold $M$. Then, the Novikov inequalities for the quotient orbifold $[M/K\ltimes M]\approx M/K$ coincide with the equivariant Novikov inequalities associated to the action of $K$ on $M$.
\end{corollary}

The previous result provides us with a simple way to obtain examples where the Novikov inequalities for certain orbifolds can be described from known ones \cite{BravermanFarber0,BravermanFarber}.

Let us exhibit other interesting situations:

\begin{example}
Let $G\rr M$ be a proper and étale Lie groupoid for which there exists a basic smooth function $f:M\to S^1$ with no critical points. If $d\theta$ is the angle form on $S^1$ then the Novikov numbers $b_j(\xi)$ and $q_j(\xi)$ for $\xi=[f^\ast(d\theta)]$ must vanish. More importantly, if we consider the product orbifold $([M/G]\times N,\overline{f^\ast(d\theta)}+\omega)$ where $N$ is a compact manifold and $\omega$ is a closed 1-form of Morse type on $N$ then we can describe the Novikov inequalities for $f^\ast(d\theta)+\omega$ out of the Novikov inequalities for $\omega$. It is worth saying that $f^\ast(d\theta)+\omega$ is a closed basic 1-form of Morse type, as it is locally given by the sum of two Morse--Bott functions which give rise to a Morse--Bott function on the product $M\times N$, see \cite[p. 736]{Latschev}. Some explicit computations for the Novikov inequalities in the classical case of compact manifolds can be found for instance in \cite[s. 6.3.7]{BanHurSpaeth}.
\end{example}

\begin{example}
Let $K$ be a connected Lie group and $M$ be a simply connected manifold. Suppose that there is a proper and locally free action of $K$ on $M$, so that the quotient space $M/K$ inherits the structure of an orbifold which can be codified by the action groupoid $K\ltimes M\rr M$. From the sequence \eqref{actionGpaths} it follows that such an action groupoid is $(K\ltimes M)$-simply connected. Therefore, as consequence of Proposition \ref{BasicTrivial} we conclude that the Novikov inequalities associated to any closed 1-form of Morse type on $M/K$ are completely determined by the orbifold Morse inequalities showed in \cite{H}. A particular instance of this situation is provided by the action of the circle $S^1\subset \mathbb{C}$ on the 3-sphere $S^3\subset \mathbb{C}^2$ given by $k\cdot (z,w)=(k^pz,k^qw)$ where $p,q$ are integers for which $\gcd(p,q)=1$.
\end{example}

\begin{example}
	If $M$ is a compact manifold with boundary then its \emph{double} $N$ can be constructed by gluing together a copy of $M$ and its mirror image along their common boundary. It follows that there is natural reflection action of $\mathbb{Z}_2$ on the manifold $N$ fixing the common boundary and the corresponding quotient space can be identified with $M$, so that $M$ has a natural orbifold structure. Therefore, we may think of the Novikov inequalities for these kinds of orbifolds as those inequalities described for manifolds with boundary in \cite{BravermanSilantyev}. A detailed study about orbifolds whose underlying space is a manifold with boundary can be found in \cite{Lange} and its quoted references.
\end{example}

\subsection{Zeros of symplectic vector fields}

We finish the paper by quickly explaining how Theorem \ref{ThmNokivovOrbifolds} could be applied to find a lower bound for the numbers of zeros of certain symplectic vector fields on symplectic orbifolds. Let us first introduce the notion of 0-symplectic groupoid \cite{hsz}. A \emph{foliation groupoid} is a Lie groupoid $G \rightrightarrows M$ whose space of objects $M$ is Hausdorff and whose isotropy groups $G_x$ are discrete for all $x\in M$. For instance, every \'etale  Lie groupoid with Hausdorff objects manifold is a foliation groupoid. The converse is not true, however every foliation groupoid is Morita equivalent to an \'etale groupoid. As shown in \cite{crainic,CraMoer}, being a foliation groupoid is equivalent to the associated Lie algebroid anchor map $\rho:A\to TM$ being injective. As a consequence, the manifold $M$ comes with a regular foliation $\mathcal{F}$ tangent to the leaves of $\mathrm{im}(\rho)\subseteq TM$. Note that if $G \rightrightarrows M$ is source-connected the leaves of $\mathcal{F}$ coincide with the groupoid orbits. We say that a basic $2$-form $\Omega$ on $M$ is \emph{nondegenerate} if $\ker(\Omega)=\mathrm{im}(\rho)\subseteq TM$. Accordingly, a $0$-\emph{symplectic groupoid} is a foliation groupoid equipped with a closed and nondegenerate basic $2$-form. It follows immediately that $(M,\Omega)$ is a pre-symplectic manifold with $\textnormal{ker}(\Omega)=T\mathcal{F}$ in the sense of \cite{LS}. Additionally, there is a left action of the product groupoid $G\times G \rightrightarrows M\times M$ on $G$ along $(s,s)$ given by $(g,h)f=gfh^{-1}$. The components of the orbits of this action define a regular foliation $\mathcal{F}_1$ of $G$ satisfying $T\mathcal{F}_1=\textnormal{ker}(ds)+\textnormal{ker}(dt)$. In particular, $\Omega$ is nondegenerate if and only if $\textnormal{ker}(s^\ast \Omega)=\textnormal{ker}(ds)+\textnormal{ker}(dt)$. As a consequence, $(G,s^\ast \Omega)$ is also a pre-symplectic manifold with $\textnormal{ker}(s^\ast \Omega)=T\mathcal{F}_1$. This notion of symplectic form is Morita invariant so that it yields a well defined notion of symplectic form over the differentiable stack $[M/G]$ presented by $G\rr M$.

Let $G\rr M$ be a foliation groupoid and let $\rho:A\to TM$ denote its Lie algebroid. The set of \emph{basic vector fields} $\mathfrak{X}_{\textnormal{bas}}(G)$ is by definition the quotient
$$\mathfrak{X}_{\textnormal{bas}}(G)=\frac{\lbrace (v_1,v_0)\in \mathfrak{X}(G)\times \mathfrak{X}(M): ds(v_1)=v_0\circ s,\ dt(v_1)=v_0\circ t \rbrace}{\lbrace (v_1,v_0): ds(v_1)=v_0\circ s,\ dt(v_1)=v_0\circ t,\ v_1\in (\ker(ds)+\ker(dt)) \rbrace}.$$

That is, a basic vector field is not strictly a vector field, but a pair of equivalence classes of vector fields. It is simple to see that a basic vector field $v= (v_1, v_0)$ is determined by its 2st component $v_0$. By Proposition 5.3.12 from \cite{hsz} we know that Morita equivalent foliation groupoids have isomorphic spaces of basic vector fields so that we may think of $\mathfrak{X}_{\textnormal{bas}}(G)$ as the space of vector field on $[M/G]$. Note that if we identify the basic forms $\Omega_{\textnormal{bas}}^\bullet(G)$ with the set of pairs $\lbrace (\theta_1,\theta_0)\in \Omega^\bullet(G)\times \Omega^\bullet(M): s^\ast(\theta_0)=\theta_1= t^\ast(\theta_0)\rbrace$ then we have contraction operations and Lie derivatives $\iota: \mathfrak{X}_{\textnormal{bas}}(G)\times \Omega_{\textnormal{bas}}^\bullet(G)\to \Omega_{\textnormal{bas}}^{\bullet-1}(G)$ and $\mathcal{L}:\mathfrak{X}_{\textnormal{bas}}(G)\times \Omega_{\textnormal{bas}}^\bullet(G)\to \Omega_{\textnormal{bas}}^{\bullet}(G)$ respectively defined by

$$\iota_v\theta=(\iota_{\tilde{v_1}}\theta_1,\iota_{\tilde{v_0}}\theta_0)\quad\textnormal{and}\quad \mathcal{L}_v\theta=(\mathcal{L}_{\tilde{v_1}}\theta_1,\mathcal{L}_{\tilde{v_0}}\theta_0),$$
where $(\tilde{v_1},\tilde{v_0})\in \mathfrak{X}(G)\times \mathfrak{X}(M)$ is a representative of $v$. These expressions do not depend on the choice of $(\tilde{v_1},\tilde{v_0})$, see \cite{hsz}.

Let us now suppose that $\Omega$ defines a $0$-symplectic structure on $G\rr M$. The nondegeneracy requirement implies that the contraction with a symplectic form $\Omega$ induces a linear isomorphism $\Omega^\flat:\mathfrak{X}_{\textnormal{bas}}(G)\to \Omega_{\textnormal{bas}}^1(G)$. We say that a basic vector field $v$ is \emph{symplectic} if $\mathcal{L}_{\tilde{v_0}}\Omega=0$. Note that after using the Cartan formula for the Lie derivative the latter requirement is equivalent to asking that $\omega=\iota_{\tilde{v_0}}\Omega$ is a closed basic $1$-form. But we already know that if $\omega$ is a closed $1$-form then the formula $\omega=\iota_{\tilde{v_0}}\Omega$ defines a basic vector field $v$ which must be symplectic since $\omega$ is closed. Therefore, it follows that there is a one-to-one correspondence between the closed basic 1-forms and basic symplectic vector fields.

Motivated by Proposition 2.6 in \cite{LM} we define the \emph{critical point set of a basic vector field} $v$ as the critical point set of $\tilde{v_0}$ viewed as a section of the vector bundle $TM/\textnormal{im}(\rho)\to M$. It is simple to check that such a definition does not depend on $\tilde{v_0}$. Because of the nondegeneracy condition imposed over $\Omega$ it holds automatically that the critical points of $v$ and $\omega=\iota_{\tilde{v_0}}\Omega$ agree. Hence, on 0-symplectic groupoids, the problem of estimating from below the number of zeros of closed basic 1-forms is equivalent to finding a lower bound for the numbers of zeros of basic symplectic vector fields. In consequence, the natural generalization of the Novikov theory we have developed in this paper provides a tool for using topological methods to study zeros of symplectic vector fields on symplectic orbifolds, which are particular cases of $0$-symplectic groupoids. Furthermore, it also opens new research directions for many important physical models which can be described by the Hamiltonian formalism over orbifolds allowing closed basic 1-forms as their Hamiltonians. This can be done in the same spirit that it was studied by Novikov in \cite{No2}.

\vspace*{0.5cm}
\noindent {\bf Conflict of interests.} The author declares that he has no conflict of interests.\\

\noindent {\bf Data availability.} The author declares that his manuscript has no associated data.

\end{document}